\documentclass[10pt]{amsart}
\usepackage{times,amsmath,amsbsy,amssymb,amscd,mathrsfs}
\usepackage{slashbox}\usepackage{graphicx,subfigure,epstopdf,wrapfig,chemarrow}

\usepackage{algorithm2e} 
\usepackage{multicol,multirow}
\usepackage{mathtools}
\usepackage[usenames,dvipsnames,svgnames,table]{xcolor}
\usepackage[all]{xy}
\usepackage{wrapfig}
\usepackage{tcolorbox}

\usepackage{tikz,tikz-cd}
\usepackage[utf8]{inputenc}
\usepackage{pgfplots} 
\usepackage{pgfgantt}
\usepackage{pdflscape}
\pgfplotsset{compat=newest} 
\pgfplotsset{plot coordinates/math parser=false}
\newlength\fwidth

\usepackage[numbered]{mcode}

\definecolor{myBlue}{rgb}{0.0,0.0,0.55}
\usepackage[pdftex,colorlinks=true,citecolor=myBlue,linkcolor=myBlue]{hyperref}

\usepackage[hyperpageref]{backref}

\usepackage{comment,enumerate,multicol,xspace}

  \newcounter{mnote}
  \let\oldmarginpar\marginpar
    \renewcommand\marginpar[1]{\-\oldmarginpar[\raggedleft\footnotesize #1]%
    {\raggedright\footnotesize #1}}

\newtheorem{theorem}{Theorem}[section]
\newtheorem{lemma}[theorem]{Lemma}

\newtheorem{remark}[theorem]{Remark}

\newcommand{\dx}{\,{\rm d}x}
\newcommand{\dd}{\,{\rm d}}

\newcommand{\bs}{\boldsymbol}

\DeclareMathOperator*{\spa}{span}

\DeclareMathOperator{\dist}{dist}

\usepackage{stmaryrd}
\usepackage{scalefnt}
\usepackage{graphicx}
\newcommand{\Oplus}{\ensuremath{\vcenter{\hbox{\scalebox{1.5}{$\oplus$}}}}}

\begin{document}
\title[Geometric Decompositions of Smooth Finite Elements]{Geometric Decompositions of The Simplicial lattice and Smooth Finite Elements in Arbitrary Dimension}

\author{Long Chen}%
 \address{Department of Mathematics, University of California at Irvine, Irvine, CA 92697, USA}%
 \email{chenlong@math.uci.edu}%
 \author{Xuehai Huang}%
 \address{School of Mathematics, Shanghai University of Finance and Economics, Shanghai 200433, China}%
 \email{huang.xuehai@sufe.edu.cn}%

\thanks{The first author was supported by NSF DMS-1913080 and DMS-2012465.}
\thanks{%The second author is the corresponding author. 
The second author was supported by the National Natural Science Foundation of China under grants 11771338 and 12171300, the Natural Science Foundation of Shanghai 21ZR1480500, and the Fundamental Research Funds for the Central Universities 2019110066.}

\makeatletter
\@namedef{subjclassname@2020}{\textup{2020} Mathematics Subject Classification}
\makeatother
\subjclass[2020]{
  %58J10;   %% Differential complexes [See also 35Nxx]; elliptic complexes
  %%65N55;   %% Multigrid methods; domain decomposition for boundary value problems involving PDEs
  %%65F10;   %%  Iterative numerical methods for linear systems
  % 65N12;   %%  Stability and convergence of numerical methods for boundary value problems involving PDEs
  % 65N15;   %%  Error bounds for boundary value problems involving PDEs
  % 65N22;   %%  Numerical solution of discretized equations for boundary value problems involving PDEs
  65N30;   %%  Finite element, Rayleigh-Ritz and Galerkin methods for boundary value problems involving PDEs
  %65F08;   %% Preconditioners for iterative methods
  74S05; %% Finite element methods applied to problems in solid mechanics
}
 \begin{abstract}
Recently $C^m$-conforming finite elements on simplexes in arbitrary dimension are constructed by Hu, Lin and Wu. The key in the construction is a non-overlapping decomposition of the simplicial lattice in which each component will be used to determine the normal derivatives at each lower dimensional sub-simplex. A geometric approach is proposed in this paper and a geometric decomposition of the finite element spaces is given. Our geometric decomposition using the graph distance not only simplifies the construction but also provides an easy way of implementation. 
 \end{abstract}

\keywords{$C^m$-conforming finite element, simplicial lattice, geometric decomposition, graph distance}

\maketitle

%\tableofcontents

\section{Introduction}
In a recent work \cite{huConstructionConformingFinite2021}, Hu, Lin and Wu have solved a long-standing open problem in finite element methods: construction of $C^m$-conforming finite elements on simplexes in arbitrary dimension. It unifies the scattered results  \cite{BrambleZlamal1970,Zenisek1970,ArgyrisFriedScharpf1968} in two dimensions, \cite{Zenisek1974a,Zhang2009a} in three dimensions, and \cite{Zhang2016a} in four dimensions. In this paper, we provide a geometric decomposition of the finite element spaces constructed in \cite{huConstructionConformingFinite2021} and consequently give a simplified construction different from~\cite{huConstructionConformingFinite2021}.

A finite element on a geometric domain $K$ is defined as a triple $(K, V, {\rm DoF})$ by Ciarlet in~\cite{Ciarlet1978}, where $V$ is the finite-dimensional space of shape functions and the set of degrees of freedom (DoFs) is a basis of the dual space $V'$ of $V$. In this paper $K = T$ is a simplex in $\mathbb R^n$ and $V = \mathbb P_k(T)$ is the polynomial space with degree bounded by $k$. The difficulty is to identify an appropriate basis of $V'$ to enforce the continuity of the functions across the boundary of the elements so that the global finite element space is a subspace of $C^m(\Omega)$, where $\Omega\subset \mathbb R^n$ admits a conforming triangulation $\mathcal T_h$ consisting of simplexes. It is well known that for a piecewise smooth function to be in $C^{m}$, it suffices to ensure the continuity of the normal derivative $\partial_n^{r} u$ for $r=0,1,\ldots, m$ across each $(n-1)$-dimensional face of $\mathcal T_h$. Those $(n-1)$-dimensional faces will meet at lower dimensional sub-simplexes. For example, in three dimensions, faces will share edges and vertices. The continuity of $\partial_n^{r} u$ on faces will imply stronger smoothness on edges and vertices,  which is known as the super-smoothness \cite{Sorokina:2010Intrinsic,Floater;Hu:2020characterization}. Indeed in \cite{huConstructionConformingFinite2021}, the authors constructed $C^m$-conforming finite elements under the following requirement on
the $C^{r_{\ell}}$ smoothness at $\ell$-dimensional sub-simplexes
$$
r_{n}=0,\;\; r_{n-1}=m\geq0,\;\; r_{\ell}\geq 2r_{\ell+1} \; \textrm{ for } \ell=n-2,\ldots, 0,
$$
and the degree of polynomial is exponential in $n$:
$$
k\geq 2r_0+1 \geq 2^n m + 1.
$$
Notice that such a requirement is sufficient but by no means necessary.  

We introduce the simplicial lattice $\mathbb T^{n}_k = \left \{ \alpha = (\alpha_0, \alpha_1, \ldots, \alpha_n)\in\mathbb N^{0:n} \mid  |\alpha | = k \right \}$ as the multi-index set with fixed length $k$. Due to the one-to-one mapping between the Bernstein polynomial $\lambda^{\alpha}$, where $\lambda$ is the barycentric coordinate, and the lattice node $\alpha\in \mathbb T^{n}_k $, the key in the construction is a non-overlapping decomposition (partition) of the simplicial lattice in which each component will be used to determine the normal derivatives at each lower dimensional sub-simplex. In  \cite{huConstructionConformingFinite2021}, a purely algebraic and combinatory approach is used to prove the existence of such partition. In this paper, a geometric approach will be proposed. 

% the construction of $C^m$-conforming finite element space on simplexes is equivalent to find an appropriate decomposition of the simplicial lattice. 
%The simplicial lattice can also be embed into a geometric simplex by using $\alpha/k$ as the barycentric coordinate of node $\alpha$.
We shall introduce a graph distance $\dist(\alpha, f)$ from a lattice node $\alpha \in \mathbb T^{n}_k$ to a $\ell$-dimensional sub-simplex $f\in \Delta_{\ell}(T)$  and reveal the following fact:
%crucial lower triangular structure of the derivative. Let $f\in \Delta_{\ell}(T)$ be a sub-simplex of $T$. 
for $\alpha\in \mathbb T^{n}_k, \beta \in \mathbb N^{1:n}$, and $\dist(\alpha, f) > |\beta|$, we have
\begin{equation}\label{eq:dist2diff}
D^{\beta} \lambda^{\alpha}|_{f} = 0.
\end{equation}

For $f \in \Delta_{\ell}(T)$, denote by $f^{*} \in \Delta_{n- \ell-1}(T)$ the sub-simplex of $T$ opposite to $f$. For a lattice node $\alpha\in \mathbb T^{n}_k$, it can be decomposed into two components $\alpha_f, \alpha_{f^*}$, and $\dist(\alpha, f) =  | \alpha_{f^*}|$. 
To give a geometric decomposition of DoFs which ensures the $C^m$-conformity, we reveal the one-to-one mapping of the space $\spa\{ \lambda^{\alpha} = \lambda_f^{\alpha_f}\lambda_{f^*}^{\alpha_{f^*}} , \alpha\in \mathbb T^{n}_k, | \alpha_{f^*}| = s \}$ with the DoFs
\begin{equation}\label{eq:intronormalDof}
\int_f  \frac{\partial^{\beta} u}{\partial n_f^{\beta}} \, \lambda_f^{\alpha_f} \dd s \quad \forall~\alpha\in \mathbb T^{n}_k, |\alpha_f| = k - s, \beta \in \mathbb N^{1:n-\ell}, |\beta | = s
\end{equation}
by mapping $\alpha_{f^*}$ to $\beta$.

Based on \eqref{eq:dist2diff} and \eqref{eq:intronormalDof}, we establish a direct decomposition of the simplicial lattice on an $n$-dimensional simplex $T$:
\begin{equation}\label{intro:smoothdecnd}
  \mathbb T^{n}_k(T) = \Oplus_{\ell = 0}^{n}\Oplus_{f\in \Delta_{\ell}(T)} S_{\ell}(f),
\end{equation}
where
\begin{align*}
S_0(\texttt{v}) &=  D(\texttt{v}, r_0), \\
S_{\ell}(f) &= D(f, r_{\ell}) \backslash \left [ \bigcup_{i=0}^{\ell-1}\bigcup_{e\in \Delta_{i}(f)}D(e, r_{i}) \right ], \; \ell = 1,\dots, n-1, \\
S_n(T) & = \mathbb T^{n}_k(T) \backslash  \left [  \bigcup_{i=0}^{n-1}\bigcup_{f\in \Delta_{i}(T)}D(f, r_{i}) \right ].
\end{align*}
Here $D(f, r) = \{ \alpha \in \mathbb T^{n}_k, \dist(\alpha,f) \leq r\}$ contains lattice nodes at most $r$ distance away from $f$ and will be called the lattice tube of $f$ with radius $r$. The requirement $r_{\ell-1}\geq 2r_{\ell}$ ensures $\{ D(f, r_{\ell}) \backslash \left [ \cup_{e\in \Delta_{\ell - 1}(f)} D(e, r_{\ell - 1})\right ], f\in \Delta_{\ell} (T)\}$ are disjoint so that \eqref{intro:smoothdecnd} is a direct decomposition. 
%An explicit characterization of $S_{\ell}(f)$ is
%\begin{equation*}
%S_{\ell}(f) =  \{\alpha\in\mathbb T_{k}^{0:n}: |\alpha_{f^*}|\leq r_{\ell}, |\alpha_{e}|\leq k - r_{i}-1, \forall e\in\Delta_i(f), i=0,\ldots,\ell-1\}.
%\end{equation*}
Geometrically we push all lattice nodes in $S_{\ell}(f)$ to the face $f$ to determine $r_{\ell}$ normal derivatives on $f$. 
Consequently we have the geometric decomposition %of $\mathbb P_{k}(T)$
\begin{equation*}%\label{intro:PrSdec}
\mathbb P_{k}(T) = \Oplus_{\ell = 0}^{n} \Oplus_{f\in \Delta_{\ell}(T)} \mathbb P_k(S_{\ell}(f)).
\end{equation*}
The direct decomposition \eqref{intro:smoothdecnd}, together with  \eqref{eq:dist2diff} and \eqref{eq:intronormalDof}, implies the DoFs
% \begin{align}
% \label{intro:C1nd0}
% D^{\alpha} u (\texttt{v}) & \quad \alpha \in \mathbb N^{1:n}, |\alpha | \leq  r_0, \texttt{v}\in \Delta_0(T),\\
% \label{intro:C1nd2}
% \int_f \frac{\partial^{\beta} u}{\partial n_f^{\beta}}  \, \lambda_f^{\alpha_f} \dd s & \quad \alpha\in S_{\ell}(f), |\alpha_f| = k - s, \beta \in \mathbb N^{1:n-\ell}, |\beta | = s, s=0,\ldots, r_{\ell},\\
% &\quad f\in \Delta_{\ell}(T), \ell =1,\ldots, n-1, \notag \\
% \label{intro:C1nd3}
% \int_T u \lambda^{\alpha} \dx & \quad \alpha \in S_n(T).
% % \in \mathbb N^{0:3}_{r-4(m+1)}, \alpha \leq r - r_0-m  - 2 \textrm{ and } \\
% %&\quad \alpha_i+\alpha_j \leq r  - r_1 -2m-3 \textrm{ for } 0\leq i<j\leq 3. \notag
% \end{align}
\begin{align}
\label{intro:C1Rnd0}
D^{\alpha} u (\texttt{v}) & \quad \alpha \in \mathbb N^{1:n}, |\alpha | \leq  r_0, \texttt{v}\in \Delta_0(\mathcal T_h),\\
\label{intro:C1Rndf}
\int_f \frac{\partial^{\beta} u}{\partial n_f^{\beta}}  \, \lambda_f^{\alpha_f} \dd s & \quad \alpha\in S_{\ell}(f), |\alpha_f| = k - s, \beta \in \mathbb N^{1:n-\ell}, |\beta | = s, \\
&\quad f\in \Delta_{\ell}(\mathcal T_h), \ell =1,\ldots, n-1, s=0,\ldots, r_{\ell},\notag \\
\label{intro:C1RndT}
\int_T u \lambda^{\alpha} \dx & \quad \alpha \in S_n(T), T\in \mathcal T_h.
% \in \mathbb N^{0:3}_{r-4(m+1)}, \alpha \leq r - r_0-m  - 2 \textrm{ and } \\
%&\quad \alpha_i+\alpha_j \leq r  - r_1 -2m-3 \textrm{ for } 0\leq i<j\leq 3. \notag
\end{align}
%We can also remove bubble functions in DoFs \eqref{intro:C1Rndf}-\eqref{intro:C1RndT} as $\alpha_f\geq r_{\ell-1}+1-|\alpha_{f^*}|$.
DoF \eqref{intro:C1Rnd0} means $u$ is $C^{r_0}$-continuous at vertices in $\Delta_0(\mathcal T_h)$.
By DoFs \eqref{intro:C1Rnd0}-\eqref{intro:C1Rndf} and the unisolvence of the finite element on faces, $\frac{\partial^{\beta} u}{\partial n_f^{\beta}}$ is single-valued across face $f\in \Delta_{\ell}(\mathcal T_h)$ for $|\beta|\leq r_{\ell}$, that is $u$ is $C^{r_{\ell}}$-continuous across $f$ and in particular $C^m$-continuous across $(n-1)$-dimensional faces as $r_{n-1}=m$. The interior DoF \eqref{intro:C1RndT} is included for the unisolvence. Our geometric decomposition using the distance not only simplifies the construction but also provide an easy way of implementation. 

%$\{S_{\ell}(f), f\in\Delta_{\ell}(T), \ell=0,\ldots,n\}$ are disjoint.
% \LC{Add DoFs to determine $C^{r_{\ell}}$ smoothness and briefly mention why it ensures the $C^m$-conformity. }

Besides standard finite elements on simplexes, $C^m$-conforming finite elements on macro-hypercubes and $C^1$-conforming finite elements on macro-simplices in arbitrary dimension are developed in \cite{HuZhang2015a} and \cite{FuGuzmanNeilan2018}, respectively. In \cite{Xu2020}, Xu exploits the artificial neural network to devise $C^m$-conforming piecewise polynomials and then develops a finite neuron method. On the other side, the lowest order nonconforming finite elements on simplexes were devised in \cite{WangXu2013,WangXu2006,WuXu2019} for $m\leq n$ and $k=m+1$.
We refer to \cite{WuXu2017,HuZhang2017,HuZhang2019} for more $H^m$-nonconforming finite elements and \cite{ChenHuang2020,Huang2020,Huang:2021-Conforming} for $H^m$-conforming and nonconforming virtual elements on any shape of polytope $K$ in $\mathbb R^n$.

%we devised $H^m$-nonconforming virtual elements of any degree $k$ on any shape of polytope $K$ in $\mathbb R^n$ with $k\geq m$ in a universal way by employing a generalized Green’s identity.
% When $K$ is a simplex, $1\leq m\leq n$ and $k = m$, the virtual elements in \cite{ChenHuang2020} are exactly the nonconforming finite elements in \cite{WangXu2013,WangXu2006}.
% And when $K$ is a simplex, $m=n+1$ and $k = m$, 
% the degrees of freedom of the virtual elements in \cite{Huang2020} are same as those of the nonconforming finite elements in \cite{WuXu2019}.
% We refer to \cite{WuXu2017,HuZhang2017,HuZhang2019} for more $H^m$-nonconforming finite elements and~\cite{ZhaoChenZhang2016,ZhaoZhangChenMao2018,AntoniettiManziniVerani2018} for more $H^m$-nonconforming virtual elements.
% \LC{Add non-conforming elements by Xu and Wang, Xu and Wu. Not sure about VEM.}

The rest of this paper is organized as follows. Some notation and simplicial lattice are introduced in Section~\ref{sec:lattice}. We show the geometric decomposition of Lagrange finite elements and Hermite finite elements in arbitrary dimension in Section~\ref{sec:lagrange} and Section~\ref{sec:hermite}, respectively.
In Section~\ref{sec:geodecomp2d} the geometric decomposition of $C^m$-conforming finite elements in two dimensions is studied. And the geometric decomposition of $C^m$-conforming finite elements in arbitrary dimension is developed in Section \ref{sec:geodecompnd}.

\section{simplicial lattice}\label{sec:lattice}
Let $T \in \mathbb{R}^{n}$ be an $n$-dimensional simplex with vertices $\texttt{v}_{0}, \texttt{v}_{1}, \ldots, \texttt{v}_{n}$ in general position. That is
$$
T = \left \{ \sum_{i=0}^n \lambda_i \texttt{v}_i \mid 0\leq \lambda_i \leq 1, \sum_{i=0}^n\lambda_i = 1 \right \},
$$
where $\lambda = (\lambda_0, \lambda_1, \ldots, \lambda_n) $ is called the barycentric coordinate. We will write $T = {\rm Convex}(\texttt{v}_0, \ldots, \texttt{v}_n)$, where ${\rm Convex}$ stands for the convex combination.

\subsection{The simplicial lattice}
For two non-negative integers $l\leq m$, we will use the multi-index notation $\alpha \in \mathbb{N}^{l:m}$, meaning $\alpha=\left(\alpha_{l}, \cdots, \alpha_{m}\right)$ with integer $\alpha_{i} \geqslant 0$. The length of a multi-index is $|\alpha|:=\sum_{i=l}^m \alpha_{i}$ for $\alpha\in \mathbb{N}^{l:m}$. We can also treat $\alpha$ as a vector with integer valued coordinates.
We define $\lambda^{\alpha}:=\lambda_{0}^{\alpha_{0}} \cdots \lambda_{n}^{\alpha_{n}}$ for $\alpha\in \mathbb{N}^{0: n}$.

A simplicial lattice of degree $k$ and dimension $n$ is a multi-index set of $n+1$ components and with fixed length $k$, i.e.,
$$
\mathbb T^{n}_k = \left \{ \alpha = (\alpha_0, \alpha_1, \ldots, \alpha_n)\in\mathbb N^{0:n} \mid  \alpha_0 + \alpha_1 + \ldots + \alpha_n = k \right \}.
$$
%Similarly define $\mathbb T^{1:n}_k =\left \{ \alpha = (\alpha_1, \ldots, \alpha_n)\in\mathbb N^{1:n} \mid \alpha_1 + \ldots + \alpha_n = k \right \}$.
An element $\alpha\in \mathbb T^{n}_k$ is called a node of the lattice. We use the convention that: for a vector $\alpha \geq c$ means $\alpha_i \geq c$ for all components $i=0,1, \ldots, n$.
It holds that
$$| \mathbb T^{n}_k | = {n + k \choose k} = \dim \mathbb P_k(T),$$
where $\mathbb P_k(T)$  denotes the set of real valued polynomials defined on $T$ of degree less than or equal to $k$.
Indeed the Bernstein basis of $\mathbb P_k(T)$ is
$$
\{ \lambda^{\alpha}: = \lambda_0^{\alpha_0}\lambda_1^{\alpha_1}\ldots \lambda_n^{\alpha_n} \mid \alpha \in \mathbb T^{n}_k\}.
$$
For a subset $S\subseteq \mathbb T^{n}_k$, we define
$$
\mathbb P_k(S) = \spa \{ \lambda^{\alpha}, \alpha \in S\subseteq \mathbb T^{n}_k \}.
$$
With such one-to-one mapping between the lattice node $\alpha$ and the Bernstein polynomial $\lambda^{\alpha}$, we can study properties of polynomials through the simplicial lattice.

Two nodes $\alpha, \beta\in \mathbb T^{n}_k$ are adjacent if there exist $0\leq i_1<i_2\leq n$ such that $|\alpha_{i_1} - \beta_{i_1}|=|\alpha_{i_2} - \beta_{i_2}|=1$ and $|\alpha_i - \beta_i|=0$ for $i\neq i_1, i_2$. By assigning edges to all adjacent nodes, the simplicial lattice becomes an undirected graph. The distance of two nodes in the graph is the length of a minimal path connecting them, where the length of a path is defined as the number of edges in the path. One can easily verify that the graph distance $\dist_G(\alpha,\beta)=\frac{1}{2}\sum_{i=0}^n|\alpha_i-\beta_i|$ which is one half of the $L^1$-norm of $\alpha - \beta$ treating $\alpha,\beta \in \mathbb R^{n+1}$.

\subsection{Geometric embedding of a simplicial lattice}
We can embed the simplicial lattice into a geometric simplex by using $\alpha/k$ as the barycentric coordinate of node $\alpha$. Given $\alpha\in \mathbb T^{n}_k$, the barycentric coordinate of $\alpha$ is given by
$$
\lambda(\alpha) = (\alpha_0, \alpha_1, \ldots, \alpha_n )/k.
$$
Let $T$ be a simplex with vertices $\{\texttt{v}_0, \texttt{v}_1, \ldots, \texttt{v}_n\}$. The geometric embedding is
$$
x: \mathbb T^{n}_k \to T, \quad x(\alpha) = \sum_{i=0}^n \lambda_i(\alpha) \texttt{v}_i. 
$$
We will always assume such a geometric embedding of the simplicial lattice exists and write as $\mathbb T^{n}_k(T)$. See Fig. \ref{fig:lattice} for an illustration for a two-dimensional simplicial lattice embedded into a triangle.

The so-called reference simplex $\hat T$ is spanned by vertices $\texttt{v}_0 = \bs 0$ and $\texttt{v}_i = \bs e_i = (0, \ldots, 1, \ldots, 0)$, whose barycentric coordinate $\lambda_i = x_i$ for $i=1,2,\ldots n$ and $\lambda_0 = 1 - \sum_{i=1}^n x_i$. If we embed $\mathbb T^{n}_k$ to the scaled reference simplex $k \hat T$, then the coordinate of $(\alpha_0, \alpha_1, \ldots, \alpha_n )\in \mathbb T^{n}_k(k \hat T)$ is simply $(\alpha_1, \ldots, \alpha_n )$, where $\alpha_0$ is dropped as the vertex $\texttt{v}_0$ is mapped to the origin. Of course, we can set other vertex as the origin and obtain other embeddings.
%By such embedding, we can extend the length operator for multi-index to barycentric coordinate:
%$$
%| \lambda | = \lambda_0 + \lambda_1 + \ldots + \lambda_n
%$$

A simplicial lattice $\mathbb T^{n}_k$ is, by definition, an algebraic set. Through the geometric embedding $\mathbb T^{n}_k(T)$, we can use operators for the geometric simplex $T$ to study this  algebraic set. For example, for a subset $S\subseteq T$, we use $\mathbb T^{n}_k(S) = \{ \alpha \in \mathbb T^{n}_k, x(\alpha)\in S\}$ to denote the portion of lattice nodes whose geometric embedding is inside $S$. The superscript ${}^n$ will be replaced by the dimension of $S$ when $S$ is a lower dimensional simplex introduced in a moment. 
%In particular,  $\mathbb T^{n}_k(\stackrel{\circ}{T})$

\subsection{Sub-simplicial lattices}
Following \cite{ArnoldFalkWinther2009}, we let $\Delta(T)$ denote all the subsimplices of $T$, while $\Delta_{\ell}(T)$ denotes the set of subsimplices of dimension $\ell$, for $0\leq \ell \leq n$. The cardinality of $\Delta_{\ell}(T)$ is $\displaystyle{n+1\choose \ell+1}$. Elements of $\Delta_0(T) = \{\texttt{v}_0, \ldots, \texttt{v}_n\}$ are $n+1$ vertices of $T$ and $\Delta_n(T) = T$.

It is the combinatory structure of the simplex that plays an important role in the construction. For a sub-simplex $f\in \Delta_{\ell}(T)$, we will overload the notation $f$ for both the geometric simplex and the algebraic set of indices. Namely $f = \{f(0), \ldots, f(\ell)\}\subseteq \{0, 1, \ldots, n\}$ and 
$$
f ={\rm Convex}(\texttt{v}_{f(0)}, \ldots, \texttt{v}_{f(\ell)}) \in \Delta_{\ell}(T)
$$
is the $\ell$-dimensional simplex spanned by the vertices $\texttt{v}_{f(0)}, \ldots, \texttt{v}_{f( \ell)}$.
%There is a one-to-one mapping
%$$
%\Sigma(0 : \ell, 0: n)\to \Delta_{\ell}(T): \sigma \to f_{\sigma}.
%$$
%With a little bit abuse of notation, we also use $f$ for the index $\sigma$. That is for an $f\in \Delta_{\ell}(T)$, we treat $f\in \Sigma(0 : \ell, 0: n)$ so that $
%f=\left[\texttt{v}_{f(0)}, \ldots, \texttt{v}_{f(\ell)}\right]\in \Delta_{\ell}(T).$
%Then
%$$
%\lambda_{f}^{\alpha_f}: = \lambda_{\sigma}^{\alpha_{\sigma}} = \lambda_{\sigma(0)}^{\alpha_{\sigma(0)}} \ldots \lambda_{\sigma(\ell)}^{\alpha_{\sigma(\ell)}} = \lambda_{f(0)}^{\alpha_{f(0)}} \ldots \lambda_{f(\ell)}^{\alpha_{f(\ell)}},
%$$
%which shows direct relation to $f$ without resort to $\sigma$. When we want to emphasize the dimension of $f$, we write $f^{\ell}$ for an $f\in \Delta_{\ell}(T)$.

If $f \in \Delta_{\ell}(T)$, then $f^{*} \in \Delta_{n- \ell-1}(T)$ denotes the sub-simplex of $T$ opposite to $f$. When treating $f$ as a subset of $\{0, 1, \ldots, n\}$, $f^*\subseteq \{0,1, \ldots, n\}$ so that $f\cup f^* = \{0, 1, \ldots, n\}$, i.e., $f^*$ is the complementary of $f$. Geometrically,
$$
f^* ={\rm Convex}(\texttt{v}_{f^*(1)}, \ldots, \texttt{v}_{f^*(n-\ell)}) \in \Delta_{n- \ell-1}(T)
$$
is the $(n- \ell-1)$-dimensional simplex spanned by vertices not contained in $f$. 

We use capital $F$ for an $(n-1)$-dimensional face of $T$ and label $F_{i}$ as the $(n-1)$-dimensional face opposite to $\texttt{v}_i$, i.e., $F_i = \{ i \}^*$ as the set and as a simplex $F_i = {\rm Convex}(\texttt{v}_{0}, \ldots, \widehat{\texttt{v}}_i, \ldots, \texttt{v}_{n})$ where $\widehat{\texttt{v}}_i$ means $\texttt{v}_i$ is removed. The sub-simplicial lattice $\mathbb T_k^{n-1}(F_0) =\{ \alpha\in \mathbb N^{1:n}, |\alpha | = k\}$ can be related to derivative $D^{\alpha}u, \alpha\in \mathbb N^{1:n}, |\alpha | = k$. 
%
%
%If $f=f_{\sigma}$ with $\sigma \in \Sigma(0:  \ell, 0: n)$, then $f^{*}=f_{\sigma^{*}}$. Again $f^*$ can be refer to $\sigma^*$.

Given a sub-simplex $f\in \Delta_{\ell}(T)$, through the geometric embedding $f \hookrightarrow T$, we define the prolongation/extension operator $E: \mathbb T^{\ell}_k \to \mathbb T^{n}_k$ as follows:
$$
E(\alpha)_{f(i)} = \alpha_{i}, i=0,\ldots, \ell, \quad \text{ and } E(\alpha)_j = 0, j\not\in f.
$$
For example, assume $f = \{ 1, 3, 4\}$, then for 
$\alpha = ( \alpha_{0},\alpha_{1},\alpha_{2})\in \mathbb T^{\ell}_k$, the extension $ E(\alpha)= (0, \alpha_{0}, 0, \alpha_{1}, \alpha_{2}, \ldots, 0).$ The geometric embedding $x(E(\alpha))\in f$ which justifies the notation $\mathbb T^{\ell}_k(f)$.
With a slight abuse of notation, for a node $\alpha_f\in \mathbb T^{\ell}_k(f)$, we still use the same notation $\alpha_f\in \mathbb T^{n}_k(T)$ to denote such extension. Then we have the following direct decomposition
\begin{equation}\label{eq:decalpha}
\alpha = E(\alpha_f) + E(\alpha_{f^*}) = \alpha_f + \alpha_{f^*}, \text{ and } |\alpha | = |\alpha_f | + | \alpha_{f^*}|.
\end{equation}
Based on \eqref{eq:decalpha}, we can write a Bernstein polynomial as
\begin{equation*}
\lambda^{\alpha} = \lambda_{f}^{\alpha_f}\lambda_{f^*}^{\alpha_{f^*}},
\end{equation*}
where $\lambda_{f}=\lambda_{f(0)} \ldots \lambda_{f(\ell)}\in \mathbb P_{\ell +1}(f)$ is the bubble function on $f$.

\begin{figure}[htbp]
\begin{center}
\includegraphics[width=5.5cm]{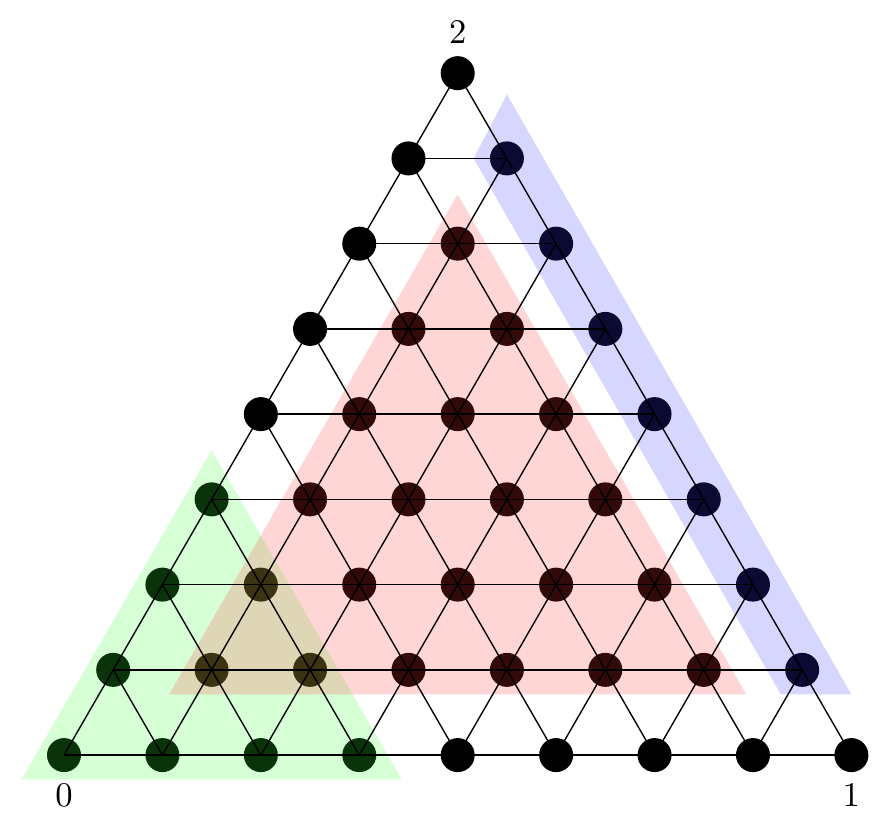}
\caption{A simplicial lattice in two dimensions.}
\label{fig:lattice}
\end{center}
\end{figure}

Define $\mathbb T^{\ell}_{k, c}(f) := \{ \alpha_{f} \in \mathbb T^{\ell}_k(f),  \alpha_{f}\geq c \}$. Then it is easy to see $\mathbb T^{\ell}_{k, 1}(f) = \mathbb T^{\ell}_{k}(\stackrel{\circ}{f})$, where the latter notation indicates the lattices nodes are contained in the interior of $f$; see the red triangle in Fig. \ref{fig:lattice}. The interior lattice $\mathbb T^{\ell}_{k}(\stackrel{\circ}{f})$
%\mnote{ Using the same kind of notation, $\mathbb T^{\ell}_{k}(\stackrel{\circ}{f})$ and $\mathbb T^{\ell}_{k-(\ell +1)}(f)$ are confused. For example, $\mathbb T^{\ell}_{k}(\stackrel{\circ}{f})\approx \mathbb T^{\ell}_{k-(\ell +1)}(f)\approx \mathbb T^{\ell}_{k-2(\ell +1)}(\stackrel{\circ}{f})\approx \mathbb T^{\ell}_{k-3(\ell +1)}(\stackrel{\circ}{f})$} 
is isomorphic to a simplicial lattice with a smaller degree $k-(\ell +1)$, denoted by $\mathbb T^{\ell}_{k-(\ell +1)}(f)$. The one-to-one mapping is 
%
%xxxx
%
%We introduce a smaller simplicial lattice associated to $\stackrel{\circ}{f}$ as 
%
%The nodes contained in the interior $\stackrel{\circ}{f}$ of $f$ form a simplicial lattice and denoted by $\mathbb T^{\ell}_{k-(\ell +1)}(f)$. It can be embedded into  $\mathbb T^{\ell}_k(f)$ as follows. 
%
%
%Then
$$
\mathbb T^{\ell}_{k-(\ell +1)}(f) \to \mathbb T^{\ell}_{k, 1}(f): \alpha_{f} \to \alpha_{f} + 1.
$$
Denote by $b_f := \lambda_f$. The interior lattice is related to the so-called bubble polynomial of $f$:
$$
b_f \mathbb P_{k-(\ell +1)}(f) := \spa\{ b_f\lambda_{f}^{\alpha_{f}}: \alpha_{f} \in  \mathbb T^{\ell}_{k-(\ell +1)}(f)\}= \spa\{ \lambda_{f}^{\alpha_{f}}: \alpha_{f} \in  \mathbb T^{\ell}_{k,1}(f)\} .
$$
Geometrically as the bubble polynomial space vanished on the boundary, it is generated by the interior lattice nodes only. In Fig. \ref{fig:lattice}, $\mathbb T^{2}_{k}(\stackrel{\circ}{T})$ consists of the nodes inside the red triangle, and $\mathbb T^{1}_{k}(\stackrel{\circ}{f})$ for $f = \{0,1\}$ is in the blue trapezoid.

In summary, by treating $f$ as a set of indices, we can apply the operators $\cup, \cap, {}^*, \backslash$ on sets. While treating $f$ as a geometric simplex, $\partial f, \stackrel{\circ}{f}$ etc can be applied.
%
%
%
% the contained nodes in $\mathbb T^{n}_k(T)$ can be written as  By removing all zeros, we can also treat $\alpha_{\sigma} \in  \mathbb N^{0:\ell}_k(f_{\sigma})$.

%The interior lattice of $f$ is \mnote{ Should $\mathcal G_{k-(\ell +1)}(\stackrel{\circ}{f})$ be $\mathcal G_{k-(\ell +1)}(f)$?}
%$$
%\mathcal G_r(\stackrel{\circ}{f}) = \{ \alpha_{\sigma} \in  \mathbb N^{0:\ell}_k: \alpha_{\sigma}\geq 1\} \cong \mathcal G_{k-(\ell +1)}(\stackrel{\circ}{f}) = ,
%$$
%and the boundary lattice is
%$$
%\mathcal G_r(\partial f) = \{ \alpha_{\sigma} \in  \mathbb N^{0:\ell}_k: \textrm{ there exists } 0\leq i\leq \ell \textrm{ such that } \alpha_{\sigma(i)} = 0\}.
%$$

\subsection{Distance to a sub-simplex}
Given $f\in \Delta_{\ell}(T)$, we define the distance of a node $\alpha\in \mathbb T_k^n$ to $f$ as
\begin{equation*}%\label{eq:dist}
\dist(\alpha, f) :=| \alpha_{f^*} | = \sum_{i\in f^*} \alpha_i.
\end{equation*}
Next we present a geometric interpretation of $\dist(\alpha, f)$.
We set the vertex $\texttt{v}_{f(0)}$ as the origin and embed the lattice to the scaled reference simplex $k\hat T$. Then $|\alpha_{f^*}| = s$ corresponds to the linear equation
$$
x_{f^*(1)} + x_{f^*(2)} + \ldots + x_{f^*(n - \ell)} = s,
$$
which defines a hyper-plane in $\mathbb R^n$, denoted by $L(f,s)$, with a normal vector $\bs 1_{f^*}$.
The simplex $f$ can be thought of as convex combination of vectors $\{\bs e_{f(0)f(i)}\}_{i=1}^{\ell}$. Obviously $\bs 1_{f^*}\cdot \bs e_{f(0)f(i)} = 0$ as the zero pattern is complementary to each other. So $f$ is parallel to the hyper-plane $L(f,s)$. The distance $\dist(\alpha, f)$ for $\alpha\in L(f,s)$ is the intercept of the hyper-plane $L(f,s)$; see Fig. \ref{fig:dist} for an illustration.
%As $\bs 0\in f_{\sigma}$ and $\dist(\bs 0, L_{\sigma}^k) = k$\mnote{ $\dist(\bs 0, L_{\sigma}^k) = \frac{k}{\sqrt{n-\ell}}$?}, we conclude $\dist(f_{\sigma}, L_{\sigma}^k) = k$.
%
%On the other hand, given $\alpha \in \mathbb N^{0:n}_{r}$, it is on the plane $L_{\sigma}^{  | \alpha_{f^*} |}$ which motivates the definition \eqref{eq:dist}.
 In particular $f\in L(f,0)$ and $\lambda_{i}|_f = 0$ for $i\in f^*$. Indeed $f = \{x\in T\mid \lambda_i(x) = 0, i\in f^*\}$.
 
%If we treat the simplicial lattice as a graph and introduce the graph distance between a node and a sub-simplex.

We can extend the definition to the distance between sub-simplexes. For $e\in \Delta_{\ell}(T),f \in \Delta(T)$, define $$\dist(e,f) = \min_{\alpha \in \mathbb T_k^{\ell}(e)} \dist (\alpha, f). $$
Then it is easy to verify that: for $e\in \Delta(f^*)$, $\dist (e,f) = k$ and for $e\in \Delta(f)$, i.e., $e\cap f \neq \varnothing,$ then $\dist (e,f) = 0$. 
%One can also show $\dist(\alpha, f)$ is the graph distance from $\alpha$ to lattice nodes in $f$.
% Namely
% $$
% \dist(\alpha, f) = \min_{\beta \in \mathbb T^{n}_k(f)} \dist_G(\alpha,\beta),
% $$
% where $\dist_G$ is the graph distance between two nodes.
%
%
%
%define \mnote{ What's the definition of $|\alpha-\beta|$? Is $|\alpha-\beta|$ the $L^1$ norm $\sum_{i=0}^n|\alpha(i)-\beta(i)|$ or $L^{\infty}$ norm $\max_{i=0}^n|\alpha(i)-\beta(i)|$? I think it is $L^{\infty}$ norm, but the $L^1$ norm is used in the proof of Lemma 1.1.}
%\mnote{ Is $\mathbb T^{n}_k$ $\mathcal G_r(f)$?}

\begin{figure}[htbp]
\subfigure[Distance to an edge.]{
\begin{minipage}[t]{0.5\linewidth}
\centering
\includegraphics*[height=4.45cm]{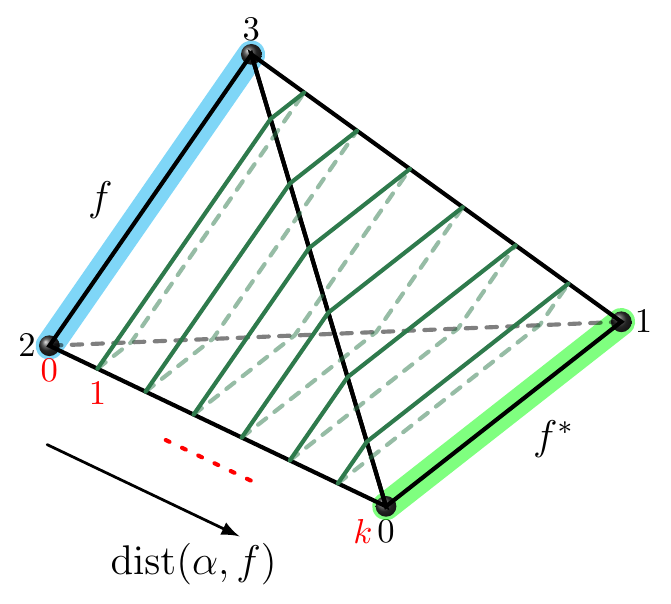}
\end{minipage}}%%
\subfigure[Distance to a face.]
{\begin{minipage}[t]{0.5\linewidth}
\centering
\includegraphics*[height=3.8cm]{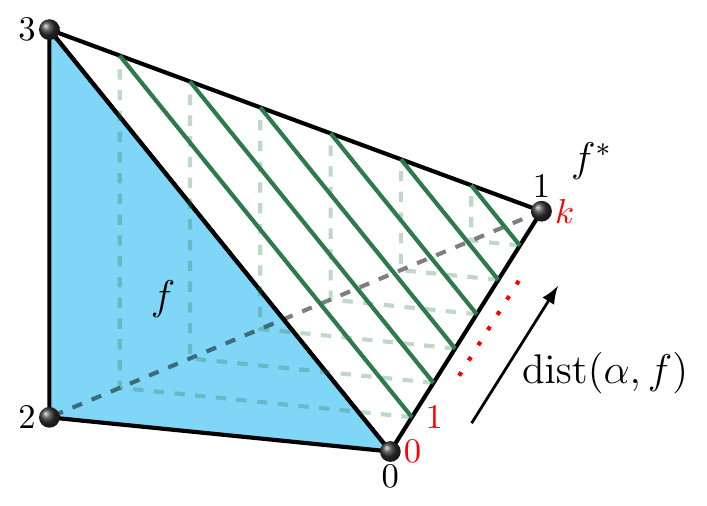}
\end{minipage}}
\caption{Distance to a sub-simplex.}
\label{fig:dist}
\end{figure}

%\begin{figure}
%\includegraphics[height=4.3cm]{figures/tetra_lattice_1.pdf}
%\includegraphics[height=3.8cm]{figures/tetra_lattice.pdf}
%\caption{Distance to a sub-simplex.}
%\label{fig:dist}
%\end{figure}

We define the lattice tube of $f$ with radius $r$ as
$$
D(f, r) = \{ \alpha \in \mathbb T^{n}_k, \dist(\alpha,f) \leq r\},
$$
which contains lattice nodes at most $r$ distance away from $f$. We overload the notation
$$
L(f,s) = \{ \alpha \in \mathbb T^{n}_k, \dist(\alpha,f) = s\},
$$
which is defined as a plane early but here is a subset of lattice nodes on this plane. Then by definition, 
$$
D(f, r) = \cup_{s=0}^r L(f,s), \quad L(f,s) = L(f^*, k - s). 
$$
We have the following characterization of $D(f, r)$.
\begin{lemma} \label{lm:dist}
%It holds that
%the nodes with equality $ \{ \alpha \in \mathbb T^{n}_k, |\alpha_{f^*}| = s \}$ is on a hyper-plane with distance $k$ to $f_{\sigma}$.
%Consequently
For lattice node $\alpha \in \mathbb T^{n}_k$,
\begin{align*}
%\label{eq:Dfr} 
\alpha \in D(f, r)  &\iff  |\alpha_{f^*}| \leq r \iff | \alpha_{f} | \geq k - r,\\
%\label{eq:notDfr} 
\alpha \notin D(f, r) & \iff  |\alpha_{f^*}| > r \iff | \alpha_{f} | \leq k - r - 1.
\end{align*}
\end{lemma}
\begin{proof}
%In particular $f_{\sigma}$ is characterized by $\{ \lambda_{\sigma^*(1)} = \lambda_{\sigma^*(2)} =\ldots \lambda_{\sigma^*(n-\ell)}= 0 \}$.
By definition of $\dist(\alpha, f)$ and the fact $| \alpha_{f} |  + | \alpha_{f^*} | = k$.
\end{proof}

For each vertex $\texttt{v}_i\in \Delta_0(T)$,
$$
D(\texttt{v}_i, r) =  \{ \alpha \in \mathbb T^{n}_k, |\alpha_{i^*}| \leq r \},
$$
which is isomorphic to a simplicial lattice $\mathbb T_{r}^n$ of degree $r$; see the green triangle in Fig. \ref{fig:lattice}. For a face $F\in \Delta_{n-1}(T)$, $D(F,r)$ is a trapezoid of height $r$ with base $F$. In general for $f\in \Delta_{\ell}(T)$, the hyper plane  $L(f, r)$ will cut the simplex $T$ into two parts, and $D(f,r)$ is the part containing $f$.

\subsection{Derivative and distance}
The distance of a node $\alpha$ to a sub-simplex $f$ can be used to control the derivative of the corresponding Bernstein polynomial.
\begin{lemma}\label{lm:derivative}
Let $f\in \Delta_{\ell}(T)$ be a sub-simplex of $T$. For $\alpha\in \mathbb T^{n}_k, \beta \in \mathbb N^{1:n}$, and $|\alpha_{f^*}| > |\beta|$, i.e., $\dist(\alpha, f) > |\beta|$, then
$$
D^{\beta} \lambda^{\alpha}|_{f} = 0.
$$
\end{lemma}
\begin{proof}
For $\alpha\in \mathbb T^{n}_k$, we write $\lambda^{\alpha} = \lambda_{f}^{\alpha_f}\lambda_{f^*}^{\alpha_{f^*}}$. When $|\alpha_{f^*}| > |\beta|$, the derivative $D^{\beta} \lambda^{\alpha}$ will contain a factor $\lambda_{f^*}^{\gamma}$ with $\gamma\in \mathbb N^{1:n-\ell},$ and $|\gamma| = |\alpha_{f^*}| - |\beta| > 0$. Therefore $D^{\beta} \lambda^{\alpha}|_{f} = 0$ as $\lambda_{i}|_f = 0$ for $i\in f^*$.
%Otherwise $|\alpha_{f^*}| \leq |\beta|$, there will be a term in $D^{\beta} \lambda^{\alpha}$ which does not contain $\lambda_{f^*}$
\end{proof}

\section{Lagrange Finite Elements}\label{sec:lagrange}

%\subsection{A geometric decomposition of Lagrange element}

For the polynomial space $\mathbb P_k(T)$ with $k\geq 1$ on an $n$-dimensional simplex $T$, we have the following decomposition of Lagrange element~\cite[(2.6)]{ArnoldFalkWinther2009}
\begin{align}
\label{eq:Prdec}
\mathbb P_k(T) &= \Oplus_{\ell = 0}^n\Oplus_{f\in \Delta_{\ell}(T)} b_f\mathbb P_{k - (\ell +1)} (f).
\end{align}
The function $u\in \mathbb P_k(T)$ is uniquely determined by DoFs
\begin{equation}\label{eq:dofPr}
\int_f u \, p \dd s \quad \forall~p\in \mathbb P_{k - (\ell +1)} (f), f\in \Delta_{\ell}(T), \ell = 0,1,\ldots, n.
\end{equation}
The integral at a vertex is understood as the function value at that vertex. 

%\subsection{A geometric decomposition of the simplicial lattice}
We shall derive \eqref{eq:Prdec} and \eqref{eq:dofPr} from a geometric decomposition of the simplicial lattice.
We use $A\oplus B$ to denote the union of two disjoint sets $A$ and $B$, i.e. $A\oplus B = A\cup B$ with property $A\cap B=\varnothing$. This notation is meant to be suggestive of the fact that  $|A\oplus B| = |A| + |B|$, where $|\cdot |$ of a set is its cardinality.

\begin{lemma}It holds that
\begin{equation}\label{eq:latticedec}
\mathbb T^{n}_k(T) = \Oplus_{\texttt{v}\in \Delta_0(T)} \mathbb T^{0} _k(\texttt{v}) \,\oplus\,\Oplus_{\ell = 1}^n \Oplus_{f\in \Delta_{\ell}(T)}\mathbb T^{\ell}_{k}(\stackrel{\circ}{f}).
\end{equation}
\end{lemma}
\begin{proof}
%Obviously 
%$\mathbb T^{\ell}_{k,1}(f)$ are disjoint sub-sets of $\mathbb T^{n}_k(T)$ as  
As $\stackrel{\circ}{f}$ are disjoint, so is $\mathbb T^{\ell}_{k}(\stackrel{\circ}{f})$. Then count the cardinality using the isomorphism $\mathbb T^{\ell}_{k}(\stackrel{\circ}{f})\cong \mathbb T^{\ell}_{k-(\ell +1)}$ to finish the proof.
\end{proof}

\begin{figure}[htbp]
\label{fig:decLH}
\subfigure[The decomposition for a Lagrange element]{
\begin{minipage}[t]{0.5\linewidth}
\centering
\includegraphics*[width=5.25cm]{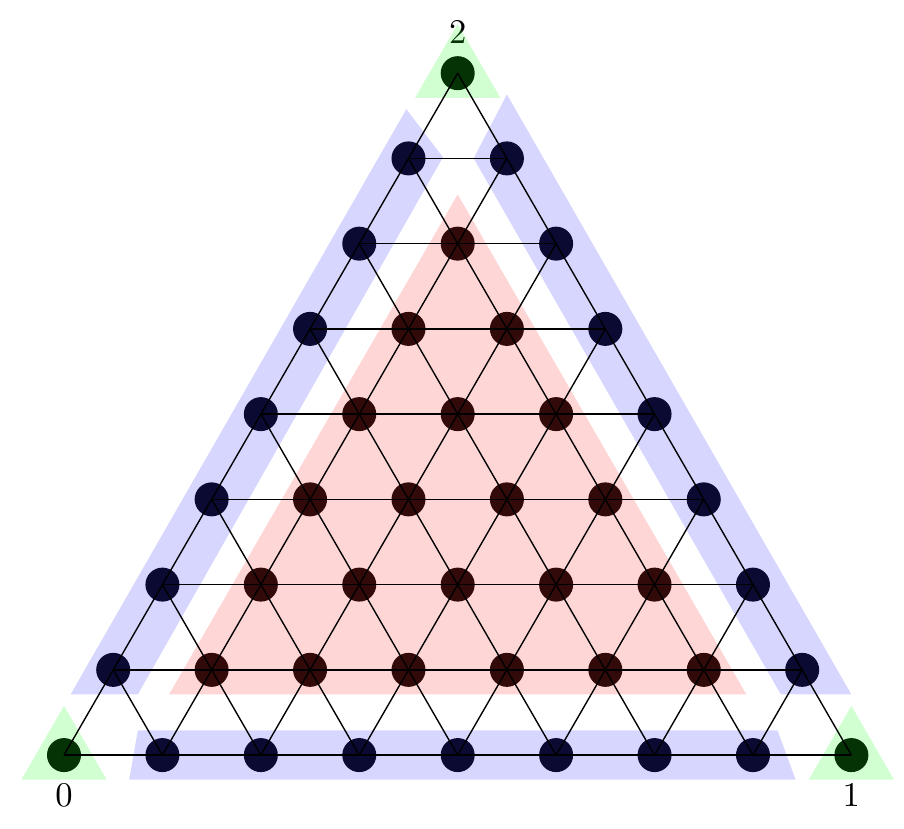}
\end{minipage}}%%
\subfigure[The decomposition for a Hermite element]
{\begin{minipage}[t]{0.5\linewidth}
\centering
\includegraphics*[width=5.25cm]{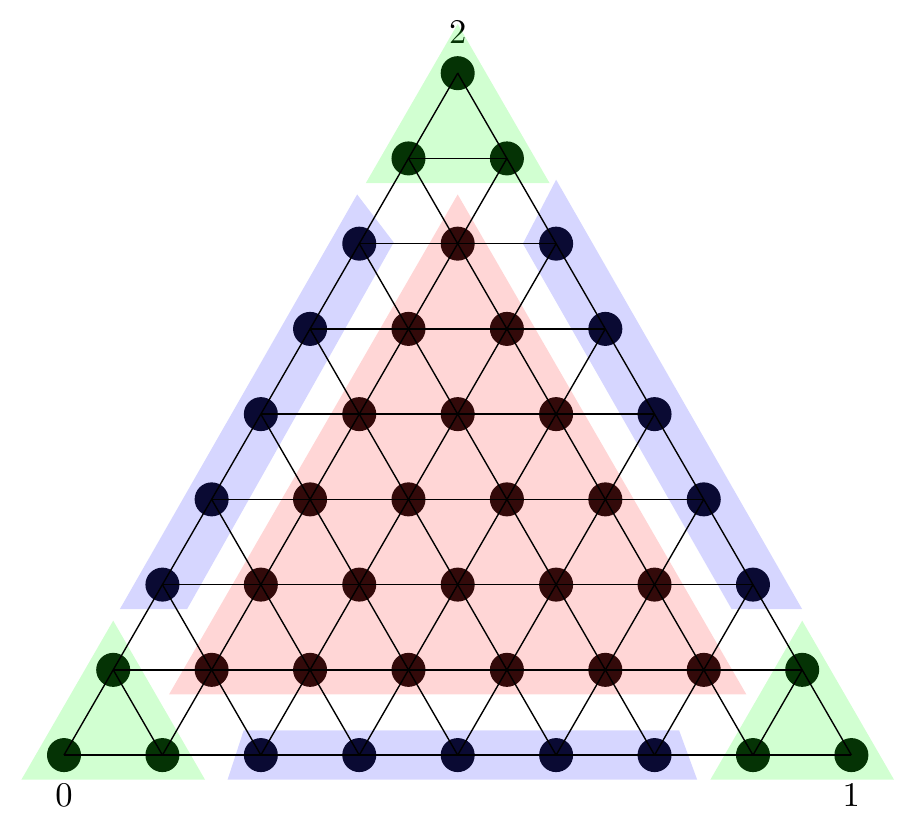}
\end{minipage}}
\caption{Decomposition of Lagrange and Hermite elements in two dimensions.}
\end{figure}

%\begin{figure}[htbp]
%\begin{center}
%\includegraphics[width=6cm]{figures/decLagrange.pdf}
%\caption{The decomposition for Lagrange element.}
%\label{fig:decL}
%\end{center}
%\end{figure}

%\subsection{Bernstein decomposition}
We can rewrite the decomposition \eqref{eq:Prdec} of polynomial and decomposition \eqref{eq:dofPr} of DoFs once we have the lattice decomposition \eqref{eq:latticedec}. The unisolvence depends on a property of the bubble function.
\begin{lemma}\label{lm:bf}
Let $f\in \Delta_{\ell}(T)$ and $b_f = \lambda_f$. For all $e\in \Delta_{m}(T)$ with $m\leq \ell$ and $e\neq f$ when $m=\ell$, then $b_f|_{e} = 0$.
\end{lemma}
\begin{proof}
We claim $f\cap e^* \neq \varnothing$. Assume $f\cap e^* = \varnothing$. Then $f^*\cup e = \{0,1,\ldots, n\}$ and thus $f\subseteq e$ which contradicts with either $m< \ell$ or $e\neq f$. 

As $f\cap e^* \neq \varnothing$, then $\lambda_f$ contains $\lambda_i$ for some $i\in e^*$ and consequently $\lambda_f|_e = 0$. 
\end{proof}

\begin{lemma}
We have the following decomposition
 \begin{equation}\label{eq:decPkT}
\mathbb P_k(T) = \Oplus_{\texttt{v}\in \Delta_0(T)} \mathbb P_k(\mathbb T^{0} _k(\texttt{v})) \,\oplus\, \Oplus_{\ell = 1}^n \Oplus_{f\in \Delta_{\ell}(T)} \mathbb P_k(\mathbb T^{\ell}_{k}(\stackrel{\circ}{f})).
%\spa\{ \lambda_i^r, i = 0,1,\ldots n\} \oplus \Oplus_{\ell = 1}^n \Oplus_{f\in \Delta_{\ell}(T)} \spa\{ \lambda_f^{\alpha}, \alpha\in \mathbb N^{0:\ell}_{k,1}(f)\}.
\end{equation}
The function $u\in \mathbb P_k(T)$ is uniquely determined by DoFs
\begin{align*}
u(\texttt{v}_i) &\qquad \texttt{v}_i \in \Delta_{0}(T),\\
\int_f u \, \lambda_f^{\alpha_f} \dd s &\qquad  \alpha_f \in \mathbb T^{\ell}_{k-(\ell + 1)}(f), f\in \Delta_{\ell}(T), \ell = 1,\ldots, n.
\end{align*}
\end{lemma}
\begin{proof}
For completeness, we adapt the unisolvence proof in \cite{ChenHuang2021} to here. We choose a basis $\{\phi_i \}$ of $\mathbb P_k(T) $ by the decomposition \eqref{eq:decPkT} and denote DoFs as $\{ N_i\}$. By construction, the dimension of $\{\phi_i \}$ matches the number of DoFs $\{ N_i\}$. The square matrix $(N_i(\phi_j))$ is block lower triangular in the sense that for $\phi_{f}\in  \mathbb P_k(\mathbb T^{\ell}_{k}(\stackrel{\circ}{f})) = b_f\mathbb P_{k - (\ell +1)} (f),$ 
$$
\int_{e}\phi_{f} p\dd s=  0, \quad \forall e \in \Delta(T)\text{ with } \dim e \leq \ell, e\neq f, p\in \mathbb P_{k-\dim e + 1}(e)
$$
by Lemma \ref{lm:bf}. The diagonal block is invertible as $b_f: \mathbb P_{k-(\ell + 1)}(f) \to b_f \mathbb P_{k-(\ell + 1)}(f)$ is an isomorphism and $b_f > 0$ in $\stackrel{\circ}{f}$. 

So the unisolvence follows from the invertibility of this lower triangular matrix. 
\end{proof}

\begin{remark}\rm
The polynomial space at vertices can be changed to $\mathbb P_1(T) = \spa\{ \lambda_i, i = 0,1,\ldots n\}$.
A hierarchical decomposition sorted by degree can be obtained which may have smaller condition number for the assembled mass and stiffness matrices. But this is not the focus of the current paper.
\end{remark}

Let $\{\mathcal {T}_h\}$ be a family of partitions
of $\Omega$ into nonoverlapping simplexes with $h_K:=\mbox{diam}(K)$ and $h:=\max_{K\in \mathcal {T}_h}h_K$.
Let $\Delta_{\ell}(\mathcal T_h)$ be the set of all $\ell$-dimensional simplexes
of $\mathcal {T}_h$ for $\ell=0, 1, \ldots, n$. The mesh $\mathcal T_h$ is conforming in the sense that the intersection of any two simplexes is a common lower sub-simplex. The global Lagrange finite element space $V_h^{\rm L}$ can be defined as
$$
V_h^{\rm L} = \Oplus_{\texttt{v}\in \Delta_0(\mathcal T_h)} \mathbb P_k(\mathbb T^{0} _k(\texttt{v})) \,\oplus\, \Oplus_{\ell = 1}^n \Oplus_{f\in \Delta_{\ell}(\mathcal T_h)} \mathbb P_k(\mathbb T^{\ell}_{k,1}(f)).
$$
Here we extend the polynomial on $f$ to each element $T$ containing $f$ by the Bernstein form in the barycentric coordinate. Consequently the dimension of $V_h^{\rm L}$ is
$$
V_h^{\rm L} = \sum_{\ell = 0}^n |\Delta_{\ell}(\mathcal T_h)| { k - 1 \choose \ell},
$$
where $|\Delta_{\ell}(\mathcal T_h)|$ is the cardinality of  number of $\Delta_{\ell}(\mathcal T_h)$, i.e., the number of $\ell$-dimensional simplices in $\mathcal T_h$.

%The bubble polynomial space
%$$
%b_f\mathbb P_{k-(\ell +1)}(f) = \spa \{ \lambda_{f}^{\alpha_f}, \alpha_f \geq 1, |\alpha_f| = k \} = \spa \{ b_f\lambda_{f}^{\alpha_f}, |\alpha_f| = k - (\ell+1) \}.
%$$
%A function $u\in b_f\mathbb P_{k-(\ell +1)}(f)$ can be determined by the DoFs \mnote{ $\mathbb N^{0:\ell}_{k}(\stackrel{\circ}{f})$ and $\mathbb N^{0:\ell}_{k-(\ell + 1)}(\stackrel{\circ}{f})$ may cause confusion. Same notation, different definition}
%\begin{equation}\label{eq:bubbleDoF}
%\int_f u \, \lambda_{f}^{\alpha_f} \,  \dd s \quad \forall~\alpha_f\in \mathbb N^{0:\ell}_{k}(\stackrel{\circ}{f}).
%\end{equation}
%As the bubble function $b_f$ is always positive inside $f$, \eqref{eq:bubbleDoF} can be written as
%\begin{equation*}
%\int_f u \, \lambda_{f}^{\alpha_f} \,  \dd s \quad \forall~\alpha_f\in \mathbb N^{0:\ell}_{k-(\ell + 1)}(\stackrel{\circ}{f}).
%\end{equation*}
%In the reduced lattice $\mathbb N^{0:\ell}_{k-(\ell +1 )}(\stackrel{\circ}{f})$, $k - (\ell + 1)$ is the true degree of the  lattice while in $\mathbb N^{0:\ell}_{k}(\stackrel{\circ}{f})$, constraint $\alpha_f\geq 1$ is imposed which corresponds to the bubble function $b_f$.

\section{Hermite Finite Elements}\label{sec:hermite}
%\LC{Seems Hermite element is only for cubic, i.e., $m = 1$ case. What is a better name?}
In this section we will show the geometric decomposition of Hermite finite elements in arbitrary dimension. When $n>1$, the Hermite finite element space is $C^0$-conforming only but $C^m$ continuous at vertices. The degree of polynomial satisfies $k \geq 2m+1$.

\subsection{Hermite spline in one dimension}
 For edge $e$ with vertices $\texttt{v}_0$ and $\texttt{v}_1$, as $k \geq 2m+1$, $D(\texttt{v}_0, m)\cap D(\texttt{v}_1, m) = \varnothing$. Recall that
$$
\mathbb T_{k, m+1}^{1}(e) = \left \{ \alpha_e \in \mathbb T^{1}_{k}(e), \alpha_e \geq m+1 \right \} = \mathbb T^{1}_{k}(e) \backslash [\Oplus_{i=0}^1D(\texttt{v}_i, m)].
$$
We then have the following decomposition of lattice $\mathbb T^{1}_k(e)$
\begin{equation*}
\mathbb T^{1}_k (e)= \Oplus_{i=0}^1D(\texttt{v}_i, m) \Oplus \, \mathbb T_{k, m+1}^{1}(e).
\end{equation*}
This leads to the decomposition of polynomial space
$$
%\mathbb P_k(e) = \Oplus_{i=0}^1\spa\{\lambda_0^{\alpha_0}\lambda_1^{\alpha_1}, (\alpha_0, \alpha_1) \in \mathbb N^{0:1}_r, |\alpha_{i^*}| \leq m\} \Oplus \, \spa\{\lambda_e^{\alpha_e}, \alpha_e \in \mathbb N^{0:1}_{k,m+1}(e)\}.
\mathbb P_k(e) = \Oplus_{i=0}^1 \mathbb P_{k}(D(\texttt{v}_i, m) ) \Oplus \mathbb P_{k}( \mathbb T^{1}_{k,m+1}(e)),
$$
and DoFs
\begin{align}\notag%\label{eq:1DHerm1}
D^{\alpha} u (\texttt{v}) & \quad \alpha \in \mathbb N, 0\leq \alpha \leq  m, \texttt{v}\in \Delta_0(e),\\
\label{eq:1DHerm2}
\int_e u \, \lambda_{e}^{\alpha_e} \,  \dd s & \quad \quad ~\alpha_e\in \mathbb T^{1}_{k,m+1}(e),
\end{align}
which is known as the Hermite spline in one dimension. Notice that as $b_e^{m+1} = \lambda_e^{m+1}$ is always positive in $\stackrel{\circ}{e}$, we can simplify the lattice set $\mathbb T^{1}_{k,m+1}(e)$ to $\mathbb T^{1}_{k - 2(m+1)}$ and replace \eqref{eq:1DHerm2} by
\begin{equation*}
\int_e u \, \lambda_{e}^{\alpha_e} \,  \dd s  \quad \quad ~\alpha_e\in  \mathbb T^{1}_{k - 2(m+1)}.
\end{equation*}

%By Lemma \ref{lm:dist}, $\alpha_e\notin D(\texttt{v}_0, m)$ implies $\alpha_1  = \alpha_{0^*}> m$ and by the same token $\alpha_0 > m$. Therefore $ \mathbb N_{k}^{0:1}(e) \backslash \cup_{i=0}^1 D(\texttt{v}_i, m) \cong \mathbb N^{0,1}_{k - 2(m+1)}$.
Given a mesh $\mathcal T_h$, the decomposition can be naturally extend to the whole mesh and DoFs are single valued at each vertex $\texttt{v}\in \Delta_0(T)$
\begin{equation} \label{eq:HermiteVh}
V^{\rm H}(\mathcal T_h) = \Oplus_{\texttt{v}\in \Delta_0(\mathcal T_h)} \mathbb P_{k}(D(\texttt{v}, m) ) \Oplus_{e\in \Delta_1(\mathcal T_h)} \mathbb P_{k}( \mathbb T^{1}_{k,m+1}(e)).
\end{equation}
As the derivative $D^{\alpha}u(\texttt{v})$ is continuous up to order $m$, $V^{\rm H}(\mathcal T_h)$ defined by \eqref{eq:HermiteVh} is a subspace of $C^{m}(\Omega)$. We shall generalize the Hermite spline to arbitrary dimension by imposing the $C^m$ continuity at vertices. Again we emphasize that, for $n>1$, the corresponding finite element space is no longer $C^m$-conforming.

\subsection{Derivatives at vertices}
Consider a function $u\in C^{m}(\Omega)$. The set of derivatives of order up to $m$ can be written as
$$
\{ D^{\alpha} u, \alpha \in \mathbb N^{1:n}, |\alpha |\leq m\}.
$$
Notice that the multi-index $\alpha \in \mathbb N^{1:n}$. We can add component $\alpha_0 = m - |\alpha|$. Then the index set forms a simplicial lattice $\mathbb T^{n}_m$ of degree $m$. For each vertex, we can use the small simplicial lattice $\mathbb T_m^{n} \cong D(\texttt{v}_i,m)$ to determine the derivatives at that vertex; see the green triangle in Fig. \ref{fig:decLH} (b) and Fig. \ref{fig:2Ddec} (a).

\begin{lemma}\label{lem:vertexunisolvence}
Let $i\in \{0,1,\ldots, n\}$. The polynomial space
\begin{equation*}%\label{eq:DvPk}
\mathbb P_k(D(\texttt{v}_i, m)): = \spa \left \{ \lambda^{\alpha}, \alpha  \in \mathbb T^{n}_k, \dist(\alpha, \texttt{v}_i) =  |\alpha_{i^*} | \leq m \right \},
\end{equation*}
is uniquely determined by the DoFs
\begin{equation}\label{eq:DvDoF}
\{ D^{\beta} u (\texttt{v}_i), \beta \in \mathbb N^{1:n}, |\beta | \leq  m\}.
\end{equation}
\end{lemma}
\begin{proof}
Obviously the dimensions match. Indeed, a one-to-one mapping is from $\alpha_{i^*}$ to $\beta$. So it suffices to show that for $u\in \mathbb P_k(D(\texttt{v}_i, m))$ if DoF \eqref{eq:DvDoF} vanishes, then $u = 0$.

Without loss of generality, we assume $i=0$. 
Clearly $\{\nabla \lambda_1, \ldots, \nabla \lambda_n\}$ forms a basis of $\mathbb R^n$.
We choose another basis $\{l^1, \ldots, l^{n}\}$ of $\mathbb R^n$, being dual to $\{\nabla \lambda_1, \ldots, \nabla \lambda_n\}$, i.e., $\nabla \lambda_{i}\cdot l^j = \delta_{i,j}$. Indeed $l^i$ is the edge vector $e_{0i}$. We can express the derivatives in this non-orthogonal basis and conclude that, for each $r=0,1,\ldots, m$, vanishing $\{D^{\beta}_nu, \beta \in \mathbb T_r^{n}(F_0)\}$ is equivalent to the vanishing $\{D^{\beta} u, \beta \in \mathbb T_r^{n}(F_0)\}$, where $D^{\beta}_nu:=\frac{\partial^{|\beta|} u}{\partial (l^1)^{\beta_1}\ldots\partial(l^n)^{\beta_n}}$. 
By the duality $\nabla \lambda_{i}\cdot l^j = \delta_{i,j}$, $i,j=1,\ldots, n$,
\begin{equation}\label{eq:Dn}
D^{\beta}_n (\lambda_{0^*}^{\alpha}) = \beta!\delta( \alpha, \beta) \quad \textrm{ for }\alpha, \beta\in \mathbb T^{n}_r(F_0),
\end{equation}
where $\beta! = \beta_1!\beta_2!\ldots \beta_n!$ and $\delta(\alpha,\beta)$ is the Kronecker delta function.

A basis of $\mathbb P_k(D(\texttt{v}_0, m))$ is given by $\{ \lambda_{0}^{k - |\alpha|}\lambda_{0^*}^{\alpha}, \alpha\in \mathbb N^{1:n}, |\alpha|\leq m \}$. The DoF-Function matrix
$(D_n^{\beta}(\lambda_{0}^{k - |\alpha|}\lambda_{0^*}^{\alpha})(\texttt{v}_0))_{\beta, \alpha}$ is block lower triangular where the lattice nodes are sorted by their length. Then if each block matrix on the diagonal ($|\alpha| = |\beta| = r$) is invertible, the whole matrix is invertible which is equivalent to the unisolvence.

%First we show that the polynomial space
%$
%\spa \left \{ \lambda^{\alpha}, \alpha  \in \mathbb T^{n}_k, |\alpha_{0^*} |=r\right \}
%$
%is uniquely determined by the DoFs
%$
%\{D^{\beta}_n u (\texttt{v}_0), \beta \in \mathbb T_r^{1:n}\}.
%$ 
%Assume $u=\sum_{\alpha \in \mathbb T_r^{1:n}}c_{\alpha}\lambda_{0}^{k - r}\lambda_{0^*}^{\alpha}$ with $c_{\alpha}\in
%\mathbb R$ and $D^{\beta}_n u (\texttt{v}_0)=0$ for all $\beta \in \mathbb T_r^{1:n}$.
%% By the duality $\nabla \lambda_{i}\cdot n^j = \delta_{i,j}$, $i,j=1,\ldots, n$,
%Apply \eqref{eq:Dn} to get
%$$
%D^{\beta}_n u (\texttt{v}_0) = \beta!  c_{\beta}  \lambda_0^{k-r}(\texttt{v}_0)=\beta!  c_{\beta} \quad\textrm{ for } \beta \in \mathbb T_r^{1:n},
%$$
%which implies $c_{\beta}=0$, and then $u=0$.

Assume $u=\sum_{\alpha \in \mathbb N^{1:n}\atop |\alpha| \leq  m}c_{\alpha}\lambda_{0}^{k - |\alpha|}\lambda_{0^*}^{\alpha}$ with $c_{\alpha}\in
\mathbb R$ and $D^{\beta} u (\texttt{v}_0)=0$ for all $\beta \in \mathbb N^{1:n}$ satisfying $ |\beta | \leq  m$. We prove $c_{\alpha} = 0$ by induction with respect to $|\alpha|$. When $|\alpha | = 0$, as $u(\texttt{v}_0)=c_{(0,\ldots,0)}$ and $u(\texttt{v}_0)=0$, we conclude $c_{(0,\ldots,0)}=0$. Assume $c_{\alpha}=0$ for all $\alpha \in \mathbb N^{1:n}$ satisfying $|\alpha|\leq r-1$, i.e., $u=\sum_{\alpha \in \mathbb N^{1:n}\atop r\leq|\alpha| \leq  m}c_{\alpha}\lambda_{0}^{k - |\alpha|}\lambda_{0^*}^{\alpha}$. By Lemma~\ref{lm:derivative}, the derivative $D^{\beta}(\lambda_{0}^{k - |\alpha|}\lambda_{0^*}^{\alpha})$ vanishes at $\texttt{v}_0$ for all $\beta\in\mathbb N^{1:n}$ satisfying $|\beta|<|\alpha|$. Hence, for $|\beta| = r$, using \eqref{eq:Dn},
$$
D_n^{\beta}u(\texttt{v}_0)=D_n^{\beta}\left(\sum_{\alpha \in \mathbb N^{1:n}, |\alpha|=r}c_{\alpha}\lambda_{0}^{k - |\alpha|}\lambda_{0^*}^{\alpha}\right)(\texttt{v}_0)=\beta!  c_{\beta} = 0,
$$
which implies $c_{\beta}=0$ for all $\beta \in \mathbb N^{1:n}$, $|\beta | = r$.
Induction for $r=1,2, \ldots, m$ to conclude $u = 0$.

\end{proof}

%We have used lattice points in $D(\texttt{v}_i, m)$ to determine the derivatives at vertices and will use the rest similar to the Lagrange element.

\subsection{A decomposition of the simplicial lattice}
%We modify the face moment by subtracting these small simplicial lattice nodes.
When $k \geq 2m + 1$, then $D(\texttt{v}, m)$ for $\texttt{v}\in \Delta_0(T)$ are disjoint. Denoted by
$$
D( \Delta_0(T), m) =  \Oplus_{\texttt{v}\in \Delta_0(T)}D(\texttt{v}, m).
$$

\begin{lemma}It holds that
 \begin{equation}\label{eq:Hermdec}
\mathbb T^{n}_k(T) = D( \Delta_0(T), m) \Oplus_{\ell = 1}^n \Oplus_{f\in \Delta_{\ell}(T)}\left [\mathbb T^{\ell}_{k,1}(f) \backslash D( \Delta_0(f), m)\right ].
\end{equation}
Consequently
\begin{align}\label{eq:HermPkdec}
\mathbb P_{k}(T) = &\mathbb P_k(D( \Delta_0(T), m))  \Oplus_{\ell = 1}^n \Oplus_{f\in \Delta_{\ell}(T)}\mathbb P_k(\mathbb T^{\ell}_{k,1}(f) \backslash D( \Delta_0(f), m)).
\end{align}
\end{lemma}
\begin{proof}
 Obviously $\mathbb T^{n}_k(T) =  D( \Delta_0(T), m) \Oplus \left [\mathbb T^{n}_k(T) \backslash D( \Delta_0(T), m) \right ]$. Then use the decomposition \eqref{eq:latticedec} for $\mathbb T^{n}_k(T)$ and the fact $\dist(\texttt{v},f) = k \geq 2m+1$ for $\texttt{v}\not\in \Delta_0(f)$ to conclude that $\mathbb T^{\ell}_{k,1}(f) \backslash D( \Delta_0(T), m) = \mathbb T^{\ell}_{k,1}(f) \backslash D( \Delta_0(f), m)$. Then the desired decomposition follows.

The decomposition of polynomial space is a consequence of Bernstein basis and  the lattice decomposition \eqref{eq:Hermdec}.
\end{proof}

\subsection{Hermite finite elements}

\begin{lemma}[Hermite element in $\mathbb R^n$]
Let $k \geq 2m + 1$ and $T$ be an $n$-dimensional simplex. The shape function space $\mathbb P_{k}(T)$ is determined by DoFs
% \mnote{why is the last one single out?}
\begin{align}\label{eq:HermDoF1}
D^{\alpha} u (\texttt{v}_i) & \quad \alpha \in \mathbb N^{1:n}, |\alpha | \leq  m, \texttt{v}_i\in \Delta_0(T), i=0,1,\ldots, n,\\
\label{eq:HermDoF2}
\int_f u \, \lambda_{f}^{\alpha_f} \,  \dd s & \quad \quad ~\alpha_f\in \mathbb T_{k-(\ell + 1)}^{\ell}, \alpha_f \leq k - m-2, f\in \Delta_{\ell}(T), \ell = 1,\ldots, n.
%,\\
%\int_T u \, \lambda^{\alpha} \,  \dd s & \quad \quad ~\alpha\in \mathbb N_{k-(n + 1)}^{0:n}(\stackrel{\circ}{T}) \backslash \cup_{i=0}^n D(\texttt{v}_i, m).
\end{align}
\end{lemma}
\begin{proof}
The proof is straight forward in view of decomposition \eqref{eq:HermPkdec}. For a polynomial $u\in \mathbb P_k(\mathbb T^{\ell}_{k,1}(f) \backslash D( \Delta_0(f), m))$, as the distance of corresponding lattice nodes to all vertices are greater than $m$, \eqref{eq:HermDoF1} vanishes which means the DoF-Fun matrix is still block lower triangular. The proof is essentially the same as that for the Lagrange element except using Lemma \ref{lem:vertexunisolvence} for lattice nodes in $D( \Delta_0(T), m)$.

The condition $k \geq 2m + 1$ is required so that the disks $D(\texttt{v}_i, m)$ are disjointed.
%\LC{ more rigorous proof here.}
The set
\begin{equation*}%\label{eq:Hconstraint}
\mathbb T^{\ell}_{k,1}(f) \backslash D( \Delta_0(T), m) = \{ \alpha_f\in \mathbb T^{\ell}_{k}(f), 1\leq \alpha_f\leq k - m - 1\},
\end{equation*}
which can be verified as follows: for any $\texttt{v}\in \Delta_0(f)$, $\alpha_{\texttt{v}}\geq 1$ as $ \alpha_f\in \mathbb T^{\ell}_{k,1}(f)$. The condition $\alpha_f\not\in D(\Delta_0(f), m)$
 is equivalent to $\dist(\alpha_f, \texttt{v}) =|\alpha_{\texttt{v}^*}| > m$ which implies the upper bound $\alpha_{\texttt{v}} \leq k - m - 1$. 
 
Then we set $\tilde \alpha_f = \alpha_f - 1$ to get the lattice set in DoF \eqref{eq:HermDoF2}. The degree of polynomial is reduced from $k$ to $k - (\ell + 1)$ as $b_f = \lambda_f\in \mathbb P_{\ell + 1}(f)$ is always positive in $\stackrel{\circ}{f}$.
\end{proof}

\begin{remark}\rm
The index set in \eqref{eq:HermDoF2} can be written in another form. For $\alpha_f\in \mathbb T_{k-(\ell + 1)}^{\ell}$, $\alpha_f\leq |\alpha_f| = k-(\ell + 1) \leq k-m-2$ if $\ell \geq m + 1$, i.e., there is no need to impose the constraint $\alpha_f\leq k-m-2$ when $\ell \geq m + 1$. Consider the case $\ell \leq m$. Take $\alpha\in \mathbb T^{\ell}_{k-(\ell +1)}(f)$. The bound $\alpha_{\texttt{v}}\leq k - m-2$ is equivalent to $|\alpha_{\texttt{v}^*}|\geq m - \ell +1$, i.e., $\alpha \notin D(\texttt{v}, m - \ell)$. Therefore
\begin{equation}\label{eq:HermDoF3}
\mathbb T^{\ell}_{k,1}(f) \backslash D( \Delta_0(f), m) \cong \mathbb T^{\ell}_{k- (\ell +1)}(f)\backslash D( \Delta_0(f), m - \ell).
\end{equation}
Geometrically we consider the inner simplicial lattice and subtract vertex disks with a smaller radiu. 
\end{remark}

\medskip

%We then refine the moment as bubble functions can be removed from DoF.

Given a mesh $\mathcal T_h$, the decomposition can be naturally extend to the whole mesh
\begin{equation*} %\label{eq:HermiteVhRn}
V^{\rm H}(\mathcal T_h) = \Oplus_{\texttt{v}\in\Delta_{0}(\mathcal T_h)} \mathbb P_k(D(\texttt{v}, m)) \,\oplus\, \Oplus_{\ell = 1}^n \Oplus_{f\in \Delta_{\ell}(\mathcal T_h)}\mathbb P_k(\mathbb T^{\ell}_{k,1}(f) \backslash D( \Delta_0(f), m)).
\end{equation*}
And DoFs are single valued at each sub-simplex (symbolically change $\Delta_{\ell}(T)$ to $\Delta_{\ell}(\mathcal T_h)$).
The obtained space $V^{\rm H}(\mathcal T_h)$ is $C^0$-conforming only for $n>1$ but $C^m$ continuous at vertices.
The dimension of $V^{\rm H}(\mathcal T_h)$ is
% \mnote{ ${m \choose \ell}$: Is it $\sum\limits_{j=\ell+1}^{m+1}{j \choose \ell-1}$? If $\alpha_{f(0)}\geq k-m-1$, then $\alpha_{f(i)}\leq k-m-2$ for $i=1,\ldots, \ell$. Hence the dimension is $\sum\limits_{\alpha_{f(0)}=k-m-1}^{k-\ell-1}{k-\alpha_{f(0)-2} \choose \ell-1}=\sum\limits_{j=\ell+1}^{m+1}{j-2 \choose \ell-1}$.}
$$
\dim V^{\rm H}(\mathcal T_h) = |\Delta_0(\mathcal T_h)| {n+m \choose m} + \sum_{\ell=1}^n|\Delta_{\ell}(\mathcal T_h)| \left [{k-1 \choose \ell} - (\ell + 1){m \choose \ell} \right ].
$$
When computing the dimension of $\mathbb P_k(\mathbb T^{\ell}_{k,1}(f) \backslash D( \Delta_0(f), m))$, it is easier to use the equivalent index set in \eqref{eq:HermDoF3}.

Compared with the Lagrange elements, more DoFs are accumulated to vertices and may reduce the dimension of the finite element space. 
%For example, in two dimensions, asymptotically $N_1 \approx 3N_0, N_2\approx 2N_0$, we have for $m=(k-1)/2$ when $k$ is sufficiently large that \mnote{ $N_0{m+2 \choose 2}+N_2\left({k-1 \choose 2}-3{m \choose 2}\right)\approx \frac{1}{4}m^2N_2 + (2m^2-\frac{3}{2}m^2)N_2=\frac{3}{4}m^2N_2=\frac{3}{16}k^2N_2$, $N_2{k-1 \choose 2}\approx\frac{1}{2}k^2N_2$. So the fraction is $\frac{3}{8}$? }
%$$
%\left[N_0{m+2 \choose 2}+N_2\left({k-1 \choose 2}-3{m \choose 2}\right)\right]
%/N_2{k-1 \choose 2}\approx \frac{3}{4}.
%$$
For example, in two dimensions, moving edge-wise and element-wise DoFs to vertices will reduce the dimension of the finite element space around one half less, which is considered as an advantage of using Hermite elements vs Lagrange elements.

\section{Smooth Finite Elements in Two Dimensions}\label{sec:geodecomp2d}
We shall re-construct the $C^m$-conforming finite element on two-dimensional triangular grids, firstly constructed by Bramble and Zl\'amal \cite{BrambleZlamal1970},  by a decomposition of the simplicial lattice. We start from a Hermite finite element space which ensures the tangential derivatives across edges are continuous. By adding degrees of freedom on the normal derivative, we can impose the continuity of derivatives across triangles. We use two-dimensional case as an introductory example for the so-called super-smoothness at lower sub-simplexes: the smoothness at vertices is $C^{2m}$ which is sufficient but may not be necessary.

We use a pair of integers $\bs r = (r_0, r_1)$ for the smoothness at $0$-dimensional sub-simplex (vertex) and at $1$-dimensional sub-simplex (edge), respectively. To be $C^m$-conforming, $r_1 = m$ is the minimum requirement for edges and $r_0\geq 2r_1$  for vertices.

\subsection{Normal derivatives}
Given an edge $e$, we first identify lattice nodes to determine the normal derivative
$$
\left \{ \frac{\partial^{\beta} u}{\partial n_e^{\beta}} \mid _e , 0\leq \beta \leq m\right \}.
$$
By Lemma \ref{lm:derivative}, if the lattice node is $r_1+1$ away from the edge, then the corresponding Bernstein polynomial will have vanishing normal derivatives up to order $r_1$.
%\begin{lemma}
%For $u = \lambda^{\alpha}, \alpha \in \mathbb N^{0:2}_k$, if $\alpha \notin D(e, k)$, then
%$$
%\frac{\partial^{k} u}{\partial n_e^{k}} \mid _e = 0.
%$$
%\end{lemma}
%\begin{proof}
%The condition $\alpha \notin D(e, k)$ means $\dist(\alpha, e)>k$. Without loss of generality, we take $e = e_{0,1}$. Then $u = \lambda_0^{\alpha_0}\lambda_1^{\alpha_1}\lambda_2^{\alpha_2}$ and $\alpha_2 > s$. Notice that edge $e_{0,1}$ is on the zero level set  of $\lambda_2$, i.e., $\lambda_2|_{e_{0,1}} = 0$. The derivative $D^k u$ will contain a factor $\lambda_2^{\alpha_2 - k}$ and thus $\alpha_2 > s$ implies $D^ku|_e = 0$.
%\end{proof}

On the two vertices, we have used $D(\Delta_0(e), r_0)$ for the derivative at vertices.
So we will use the rest, i.e., $D(e,r_1)\backslash  D(\Delta_0(e), r_0)$ for the normal derivative.

\begin{lemma}\label{lem:edgeunisolvence2d}
Let $r_0\geq r_1 \geq 0$ and $k \geq 2r_0+1$. Let $e\in \Delta_1(T)$ be an edge of a triangle $T$.
The polynomial function space $\mathbb P_{k}( D(e, r_1)\backslash D(\Delta_0(e), r_0))$ is determined by DoFs
\begin{equation*}%\label{eq:normalDoF2D}
\int_e \frac{\partial^{\beta} u}{\partial n_e^{\beta}}  \, \lambda_e^{\alpha_e} \dd s \quad \alpha_e \in \mathbb T^{1}_{k - 2(r_0+1) + \beta}, \beta = 0,1,\ldots, r_1.
\end{equation*}
\end{lemma}
\begin{proof}
Without loss of generality, we take $e = e_{0,1}$. By definition $ D(e,r_1) = \Oplus_{i=0}^{r_1} L(e,i)$, where recall that
$$
L(e,i) = \{ \alpha \in \mathbb T^{2}_k, \dist(\alpha,e) = i\} = \{ \alpha \in \mathbb T^{2}_k, \alpha_2 = i\} = \{ \alpha \in \mathbb T^{2}_k, \alpha_0+\alpha_1 = k - i\}
$$ 
consists of lattice nodes parallel to $e$ and with distance $i$. Then $L(e,i) \cong \mathbb T^{1}_{k-i}(e)$ by keeping $(\alpha_0, \alpha_1)$ only.

Now we use the requirement $\alpha \notin D(\Delta_0(e), r_0)$ to figure out the range of the nodes. Using Lemma \ref{lm:dist}, we derive from $\dist (\alpha, \texttt{v}_0) > r_0$ that $ \alpha_0 < k - r_0$. Together with $\alpha_0+\alpha_1 = k - i$, we get the lower bound $\alpha_1\geq r_0 - i+1$. Similarly $\alpha_0\geq r_0 - i+1$.
Therefore the line segment
$$
L(e,i) \backslash D(\Delta_0(e), r_0) =\{ (\alpha_0, \alpha_1, i), \alpha_0+\alpha_1 = k - i, \min\{\alpha_0, \alpha_1\} \geq  r_0 - i+1\},
$$
which can be identified with the lattice $\mathbb T^{1}_{k - 2(r_0+1) + i}$ without inequality constraint.

Applying the same argument in the proof of Lemma~\ref{lem:vertexunisolvence},
it follows from Lemma~\ref{lm:derivative} that matrix 
$(\frac{\partial^{\beta}}{\partial n_e^{\beta}}(\lambda_{e}^{\alpha_e}\lambda_{2}^{i})|_e)_{\beta, i}$ is lower triangular. 
Hence it suffices to prove the polynomial function space $\mathbb P_{k}(L(e,i) \backslash D(\Delta_0(e), r_0))$ is determined by DoFs
\begin{equation}\label{eq:normalDoF2Di}
\int_e \frac{\partial^{i} u}{\partial n_e^{i}}  \, \lambda_e^{\alpha_e} \dd s \quad \alpha_e \in \mathbb T^{1}_{k - 2(r_0+1) + i}.
\end{equation}
Take a $u= \sum_{\alpha_e \in \mathbb T^{1}_{k - 2(r_0+1) + i}}c_{\alpha_e}\lambda_e^{\alpha_e}\lambda_e^{r_0 - i+1}\lambda_2^{i}  \in \mathbb P_{k}(L(e,i) \backslash D(\Delta_0(e), r_0))
$ with coefficients $c_{\alpha_e}\in\mathbb R$. Then
$$
\frac{\partial^{i} u}{\partial n_e^{i}}|_e=i!(n_e\cdot\nabla\lambda_2)^i\lambda_e^{r_0 - i+1}\sum_{\alpha_e \in \mathbb T^{1}_{k - 2(r_0+1) + i}}c_{\alpha_e}\lambda_e^{\alpha_e}.
$$
Noting that $n_e\cdot\nabla\lambda_2$ is a constant and $\lambda_e$ is always positive in the interior of $e$, the vanishing DoF \eqref{eq:normalDoF2Di} means $c_{\alpha_e} = 0$ for all $\alpha_e\in \mathbb T^{1}_{k - 2(r_0+1) + i}$.
%\LC{Using the previous argument for Hermite. Lower triangular structure indexed by $\beta$ and $\alpha_{\texttt{v}}$. Write $u = \lambda_e^{\alpha_e} \lambda_{\texttt{v}}^{\alpha_{\texttt{v}}}$ where $\texttt{v} = e^*$. Then only consider the block diagonal $\beta = \alpha_{\texttt{v}}$ for which the normal derivative will contain $\lambda_e^{\alpha_e}$ only. Then use the range.}
%Let $w = \frac{\partial^i u}{\partial n^i}\in \mathbb P_{k - i}(T)$ for $i=0,1\ldots, r_1$. On edge $e$, by the Hermite interpolation \eqref{eq:1DHerm1}-\eqref{eq:1DHerm2}, $w|_e$ is uniquely determined by
%\begin{align*}
%\partial_e^{(j)} w(\texttt{v}) & \quad  j = 0,1,\ldots, r_0 - i, \texttt{v}\in \Delta_0(e),\\
%\int_e w \, \lambda_e^{\alpha_e}\dd s & \quad \alpha_e \in \mathbb N^{0:1}_{k - 2(r_0+1) + i} (\stackrel{\circ}{e}).
%\end{align*}
%For $u\in \spa\{ \lambda^{\alpha}, \alpha \in D(e,k)\backslash D(\Delta_0(e), r_0) \}$, as $\dist(\alpha, \Delta_0(e)) > r_0$, $D^{\beta}u(\texttt{v}) = 0$ for all $\beta\leq r_0$ and $\texttt{v}\in \Delta_0(e)$. Therefore the vanishing DoF implies $w|_e = 0$.
\end{proof}

Geometrically we push all lattice nodes in $D(e,r_1)\backslash D(\Delta_0(e), r_0)$ to the edge to determine normal derivatives on $e$ up to order $r_1$.

\subsection{The geometric decomposition}
The requirement $r_1\leq r_0$ in Lemma \ref{lem:edgeunisolvence2d} is due to the fact that the smoothness of the normal derivative is less than or equal to that of the vertices. In a triangle, a vertex will be shared by two edges and to have enough lattice nodes for each edge,  $r_0\geq 2r_1$ is required.

%\begin{figure}[htbp]
%\begin{center}
%%\includegraphics[width=4in]{figures/2Ddecomposition.png}
%\includegraphics[width=2.5in]{figures/2DC1element.png}
%\caption{A .}
%\label{fig:2Ddec}
%\end{center}
%\end{figure}

\begin{figure}[htbp]
\subfigure[The geometric decomposition of a Hermite element.]{
\begin{minipage}[t]{0.5\linewidth}
\centering
\includegraphics*[width=4.75cm]{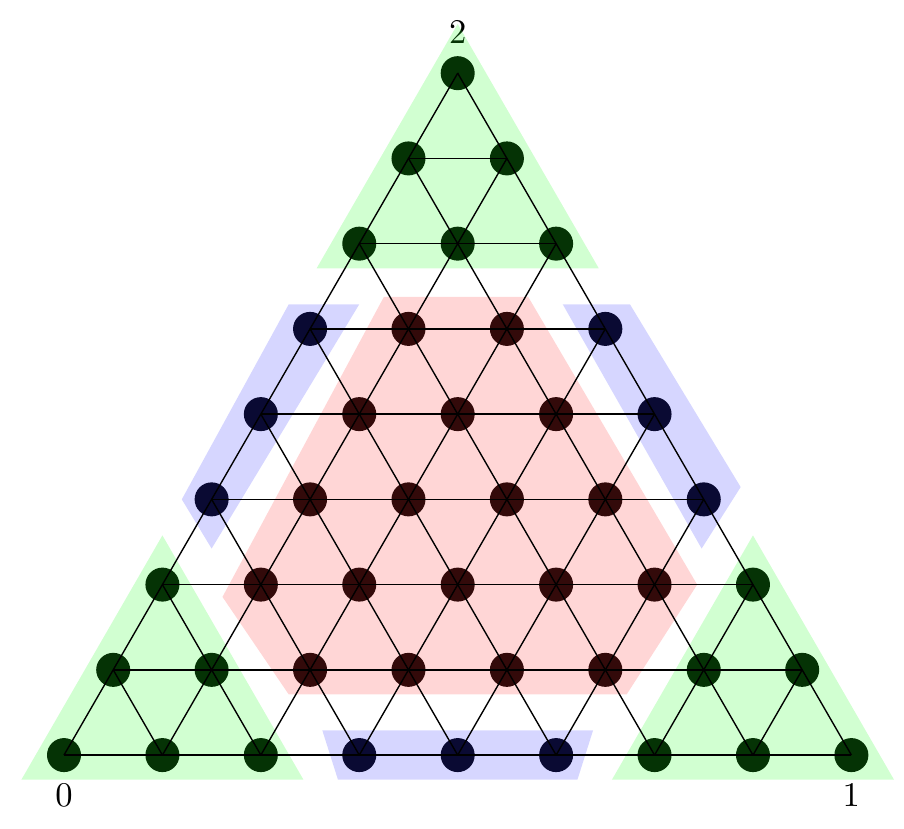}
\end{minipage}}%%
\subfigure[The geometric decomposition of a $C^1$ element: $m=1, r_1 = 1, r_0 = 2, k = 8, n = 2$.]
{\begin{minipage}[t]{0.5\linewidth}
\centering
\includegraphics*[width=4.75cm]{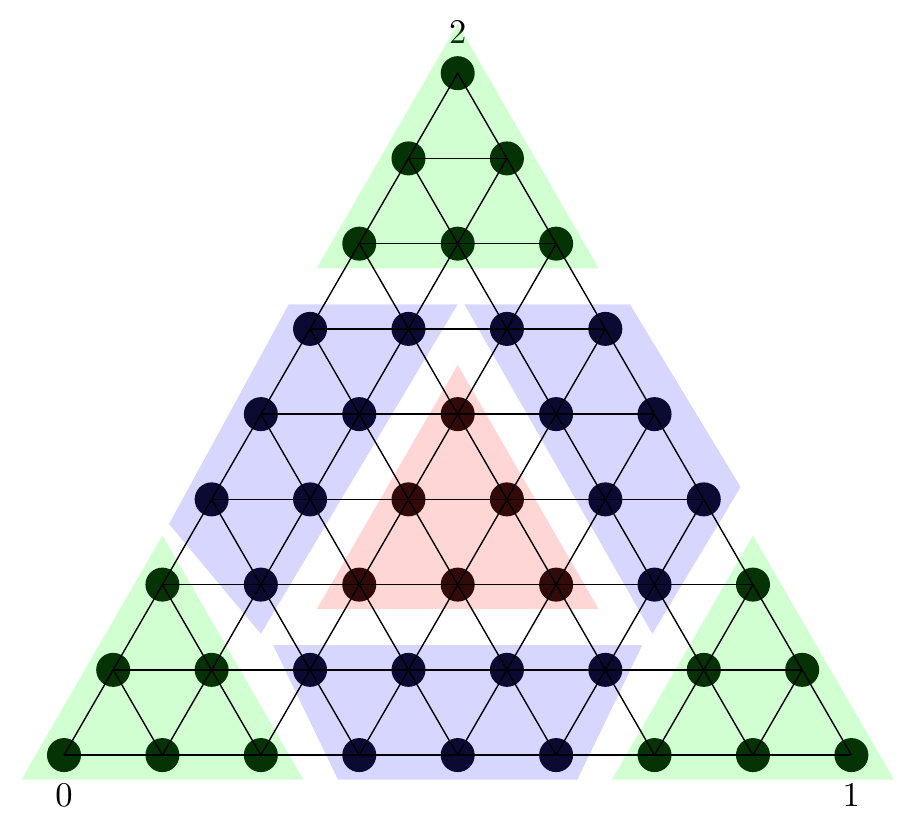}
\end{minipage}}
\caption{Comparison of the geometric decompositions of a two-dimensional Hermite element and a $C^1$-conforming element.}
\label{fig:2Ddec}
\end{figure}

\begin{lemma}\label{lem:geodecomp2d}
Let $r_1 = m$, $r_0\geq 2r_1$, and $k\geq 2r_0+1\geq 4m+1$. Let $T$ be a triangle. Then it holds that
\begin{align}\label{eq:smoothdec2d}
  \mathbb T^{n}_k(T) = S_0(T) \Oplus S_1(T) \Oplus S_2(T),
\end{align}
where
\begin{align*}
S_0(T) &=  D( \Delta_0(T), r_0), \\
S_1(T) &= \Oplus_{e\in \Delta_{1}(T)}\left ( D(e,r_1) \backslash S_0(T)\right ),\\
S_2(T) & = \mathbb T^{n}_k(T) \backslash (S_0(T) \oplus S_1(T)).
\end{align*}
This leads to the decomposition of the polynomial space
\begin{align}\label{eq:PkCmdec}
\mathbb P_k(T)&= \mathbb P_k(S_0(T)) \Oplus \mathbb P_k(S_1(T)) \Oplus \mathbb P_k(S_2(T)).
%\Oplus_{\ell = 0}^2 \spa \{ \lambda^{\alpha}, \alpha\in S_{\ell}(T)\}\\
%&= \Oplus_{i=0}^2\spa\{ \lambda^{\alpha}, \alpha\in \mathbb N^{0:2}_k, |\alpha_{i^*}|\leq s_{0}\}  \\
%&\oplus  \\
%&\oplus b_T^{m+1}\mathbb P_{k-3(m+1)}^{0,0}(T),
\end{align}
\end{lemma}
\begin{proof}
As $k\geq 2r_0+1$, the sets $\{D(\texttt{v}, r_0), \texttt{v}\in \Delta_0(T)\}$ are disjoint. 
As $\dist(\texttt{v},e) = k \geq 2r_0+1 > r_1$ for $\texttt{v}\not\in \Delta_0(e)$, $D(e,r_1) \backslash D( \Delta_0(T), r_0) = D(e, r_1) \backslash D( \Delta_0(e), r_0)$. 

We then show that the sets $\{D(e,r_1) \backslash D( \Delta_0(e), r_0), e\in \Delta_1(T)\}$ are disjoints.
%, where $S_0(e) =  D( \Delta_0(e), r_0)$. Take two edges $e_{01}$ and $e_{02}$. 
A node $\alpha \in D(e_{01}, r_1)$ implies $\alpha_2\leq r_1$ and $\alpha \in D(e_{02}, r_1)$ implies $\alpha_1\leq r_1$. Therefore $|\alpha_{0^*}| = \alpha_1+ \alpha_2 \leq 2r_1\leq r_0$, i.e., $(D(e_{01}, r_1)\cap D(e_{02}, r_1))\subseteq D(\texttt{v}_0, r_0)$. Repeat the argument for each pair of edges to conclude $\{D(e,r_1) \backslash D( \Delta_0(e), r_0), e\in \Delta_1(T)\}$ are disjoint. Then \eqref{eq:smoothdec2d} follows.
%
%As a consequence $$\mathbb P_k(T)= \Oplus_{\ell = 0}^2 \spa \{ \lambda^{\alpha}, \alpha\in S_{\ell}(T)\}.$$ 
%
\end{proof}

We give another description of $S_{\ell}(T)$ to rewrite the index set. For a lattice node $\alpha \in D(e_{01},r_1) \backslash D( \Delta_0(e_{01}), r_0)$, it satisfies the constraint
$$
\alpha_2 \leq r_1, \; \alpha_0+\alpha_2 > r_0, \; \alpha_1+\alpha_2 > r_0, \;\alpha_ 0 + \alpha_1 + \alpha_2 = k.
$$
We let $\alpha_2 = i$ for $i=0,1,\ldots, r_1 = m$. Then 
%\mnote{ check the index range. seems not correct. e.g. $i=0, k- 2(r_0+1)$ could be negative. 
%
%It's OK. Because $k- 2(r_0+1)\geq-1$, $k- 2(r_0+1)=-1$ means there is no moment dof.
%}
$$
\lambda_0^{\alpha_0}\lambda_1^{\alpha_1}\lambda_2^{\alpha_2} = \lambda^{i}(\lambda_0\lambda_1)^{r_0 - 2i +1} \lambda_{0}^{\alpha_ 0 -  (r_0 + 1)+i}\lambda_{1}^{\alpha_1 - (r_0 + 1)+i} = b_T^i b_{e}^{r_0 - 2i+1}\lambda_e^{\alpha_e}
$$
with $\alpha_e\in \mathbb N^{0:1}_{k-2(r_0+1)+i}$. So we have 
\begin{equation}\label{eq:PkS1}
\mathbb P_k(S_1(T)) = \Oplus_{e\in\Delta_{1}(T)}\Oplus_{i=0}^mb_T^ib_e^{r_{0}+1-2i}\mathbb P_{k-2(r_{0}+1)+i}(e).
\end{equation}
For a node $\alpha\in S_2(T)$, it satisfies the constraint
$$
\alpha_ 0 + \alpha_1 + \alpha_2 = k,\; \alpha > r_1 = m, \; \alpha < k-r_0.
$$
Then set $\tilde \alpha = \alpha - (m+1)$. We can write
$$
\lambda^{\alpha} = \lambda^{m+1} \lambda^{\tilde \alpha}, \; |\tilde \alpha | = k - 3(m+1), \; \tilde \alpha\leq k-r_0 - m - 2.
$$
Therefore 
\begin{equation}\label{eq:PkS2}
\mathbb P_k(S_2(T)) =  b_T^{m+1}\mathbb P_{k-3(m+1)}^{0,0}(T),
\end{equation}
where
\begin{align*}
\mathbb P_{k-3(m+1)}^{0,0}(T):=\spa\big\{\lambda^{\alpha} \mid \alpha\in\mathbb T^{2}_{k-3(m+1)}, \alpha \leq k-r_0-m-2\big\}.
\end{align*}
When $r_0 = 2r_1 = 2m$ or $r_0 = 2r_1+1 = 2m+1$, $\mathbb P_{k-3(m+1)}^{0,0}(T) = \mathbb P_{k-3(m+1)}(T)$ as the constraint automatically holds.

%\begin{remark}\rm
Such simplification is not needed in implementation. Distance to a vertex or an edge is computable and a logic array can be used to represent $S_{\ell}(T)$.
%\end{remark}
\subsection{Smooth finite elements in two dimensions}
\begin{theorem}
 Let $r_1 = m$, $r_0\geq 2r_1$, and $k\geq 2r_0+1\geq 4m+1$. Let $T$ be a triangle. The shape function space $\mathbb P_{k}(T)$ is determined by the DoFs
\begin{align}
\label{eq:C12d0}
D^{\alpha} u (\texttt{v}) & \quad \alpha \in \mathbb N^{1:2}, |\alpha | \leq  r_0, \texttt{v}\in \Delta_0(T),\\
\label{eq:C12d1}
\int_e \frac{\partial^{\beta} u}{\partial n_e^{\beta}}  \, \lambda_e^{\alpha} \dd s & \quad \alpha \in \mathbb T^{1}_{k - 2(r_0+1) + \beta}, \beta = 0,1,\ldots, r_1,\\
\label{eq:C12d2}
\int_T u \lambda^{\alpha} \dx & \quad \alpha \in \mathbb T^{2}_{k-3(m+1)}, \alpha \leq k - r_0 - m - 2.
\end{align}
\end{theorem}
\begin{proof}
By the decomposition \eqref{eq:PkCmdec} of $\mathbb P_k(T)$ and characterization \eqref{eq:PkS1}-\eqref{eq:PkS2}, the dimension of $\mathbb P_k(T)$ matches the number of DoFs. Let $u\in\mathbb P_{k}(T)$ satisfying all the DoFs \eqref{eq:C12d0}-\eqref{eq:C12d2} vanish.
Thanks to Lemma~\ref{lem:vertexunisolvence}, Lemma~\ref{lem:edgeunisolvence2d} and Lemma~\ref{lem:geodecomp2d}, it follows from the vanishing DoFs \eqref{eq:C12d0} and \eqref{eq:C12d1} that $u\in \mathbb P_{k}(S_2(T))$.
As $b_T$ is always positive in the interior of $T$, $u=0$ holds from the vanishing DoF \eqref{eq:C12d2}.
% With DoFs \eqref{eq:C12d0} and \eqref{eq:C12d1} vanishes, only $\lambda^{\alpha}, \alpha \in \mathbb N^{0:2}_k,$ with property $\dist(\alpha, \Delta_0(T)) > r_0$ and $\dist(\alpha, \Delta_1(T)) > r_1=m$ are left to be determined.
% The condition $\dist(\alpha, \Delta_1(T)) > m$ is equivalent to $\alpha \geq m + 1$ and the condition  $\dist(\alpha, \Delta_0(T)) > r_0$ is equivalent to $\alpha \leq r - r_0 - 1$. Let $\tilde \alpha = \alpha - (m+1)$. Then
% $\lambda^{\alpha} = b_T^{m+1} \lambda^{\tilde \alpha}$. As $b_T$ is always positive, it can be removed from DoFs and \eqref{eq:C12d2} follows.
\end{proof}

When $r_1=m = 1$ and $r_0=2$, this is known as Argyris element \cite{ArgyrisFriedScharpf1968,MorganScott1975}. 
When $r_1=m$, $r_0=2m$ and $k = 4m+1$, $C^m$-continuous finite elements are constructed in \cite{BrambleZlamal1970,Zenisek1970}, whose interior DoFs are different from \eqref{eq:C12d2}. In our notation, the edge and interior DoFs in \cite{huConstructionConformingFinite2021} are adopted as
\begin{align*}
\int_e \frac{\partial^{\beta} u}{\partial n_e^{\beta}}  \, \lambda_e^{\alpha} \dd s & \quad \alpha \in \mathbb T^{1}_{k -\beta}, \alpha\geq r_0-\beta+1, \beta = 0,1,\ldots, r_1,\\
\int_T u \lambda^{\alpha} \dx & \quad \alpha \in S_2(T),
\end{align*}
which are slightly different from \eqref{eq:C12d1}-\eqref{eq:C12d2}. We further remove the edge and element bubbles in the test function space in DoFs.

With mesh $\mathcal T_h$, define the global $C^m$-continuous finite element space 
\begin{align*} 
V(\mathcal T_h) &= \{u\in C^m(\Omega): u|_T\in\mathbb P_k(T)\textrm{ for all } T\in\mathcal T_h, \\
&\qquad\textrm{ and all the DoFs \eqref{eq:C12d0} and \eqref{eq:C12d1} are single-valued}\}. 
\end{align*}
Then $V(\mathcal T_h)$ admits the following geometric decomposition 
\begin{align*} %\label{eq:Vh2d}
V(\mathcal T_h) &= \Oplus_{\texttt{v}\in \Delta_{0}(\mathcal T_h)} \mathbb P_k(D(\texttt{v}, r_0)) \,\oplus\, \Oplus_{e\in \Delta_{1}(\mathcal T_h)}\mathbb P_k\left ( D(e,r_1) \backslash S_0(e)\right ) \\
&\quad\,\oplus\,  \Oplus_{T\in \mathcal T_h}\mathbb P_k(S_2(T)), %\notag
\end{align*}
where $S_0(e) =  D( \Delta_0(e), r_0)$.
The dimension of $V(\mathcal T_h)$ is
\begin{align*}
\dim V(\mathcal T_h) &= |\Delta_0(\mathcal T_h)| {r_0+2 \choose 2} + |\Delta_{1}(\mathcal T_h)|(m+1) \left(k-2r_0-1+m/2\right) \\
&\quad+|\Delta_2(\mathcal T_h)|\left[{k-3m-1 \choose 2}-3{r_0-2m \choose 2}\right].
\end{align*}
In particular, for the minimum degree case: $r_1=m, r_0= 2m, k = 4m+1$, we denoted by $V^{\rm BZ}(\mathcal T_h)$ and the dimension
$$
\dim V^{\rm BZ}(\mathcal T_h) = |\Delta_0(\mathcal T_h)| {2m + 2 \choose 2} + |\Delta_{1}(\mathcal T_h)| {m+1 \choose 2} + |\Delta_2(\mathcal T_h)| {m \choose m - 2}.
$$
When $m\leq 1$, there is no interior moments as $k = 4m+1$ is small.

\section{Smooth Finite Elements in Arbitrary Dimension}\label{sec:geodecompnd}
In this section we shall generalize the construction to arbitrary dimension. The smoothness at sub-simplexes is exponentially increasing as the dimension decreases
$$
r_{n}=0,\;\; r_{n-1}=m,\;\; r_{\ell}\geq 2r_{\ell+1} \; \textrm{ for } \ell=n-2,\ldots, 0.
$$
And the degree of polynomial $k\geq 2r_0+1 \geq 2^n m + 1$. The key in the construction is a non-overlapping decomposition of the simplicial lattice in which each component will be used to determine the normal derivatives. 

%\LC{Add more literature review later on.}
When $n=3, m = 1, r_1 = 2, r_0 = 4$ and $k \geq 9$, it is the $C^1$ element on tetrahedron constructed by Zhang in \cite{Zhang2009a}.
When $n=4, m = 1, r_2 = 2, r_1 = 4, r_0 = 8$ and $k \geq 17$, it is the $C^1$ element on simplex in four dimensions constructed by Zhang in 
\cite{Zhang2016a}. Neilan's Stokes element~\cite{Neilan2015} is a $C^0$ element with parameters $n = 3, m=r_2 =0, r_1 = 1, r_0 = 2,$ and $k\geq 5$.

\subsection{A decomposition of the simplicial lattice}
We explain the requirement $r_{\ell-1}\geq 2r_{\ell}$. 
\begin{lemma}\label{lm:disjoint}
Let $T$ be an $n$-dimensional simplex. For $\ell = 1,\ldots, n-1$, if $r_{\ell-1}\geq 2r_{\ell}$, the sub-sets $\{ D(f, r_{\ell}) \backslash \left [ \cup_{e\in \Delta_{\ell - 1}(f)} D(e, r_{\ell - 1})\right ], f\in \Delta_{\ell} (T)\}$ are disjoint.
\end{lemma}
\begin{proof}
Consider two different sub-simplices $ f, \tilde f \in \Delta_{\ell} (T)$. The dimension of their intersection is at most $\ell - 1$. Therefore $f\cap \tilde f\subseteq e$ for some $e\in \Delta_{\ell -1}(f)$. Then $e^*\subseteq (f\cap \tilde f )^* = f^*\cup \tilde f^*$. For $\alpha \in D(f, r_{\ell})\cap D(\tilde f, r_{\ell})$, we have $|\alpha_{e^*}| \leq |\alpha_{f^*}| + |\alpha_{\tilde f^*}|\leq 2r_{\ell}\leq r_{\ell - 1}$. Therefore we have shown the intersection region $D(f, r_{\ell})\cap D(\tilde f, r_{\ell})\subseteq \cup_{e\in \Delta_{\ell - 1}(f)} D(e,r_{\ell-1})$ and the result follows.
%
% and the sets $\{ D(f,r_{\ell})\backslash \cup_{e\in \Delta_1(f)} D(e,r_1),  f\in \Delta_1(T)\}$ are disjoints so that we use different lattice nodes for normal derivatives on different faces.
\end{proof}

Next we remove $D(e, r_{i})$ from $D(f, r_{\ell})$ for all $e\in \Delta_{i}(T)$ and $i=0,1,\ldots, \ell-1$. 
\begin{lemma} \label{lm:Deltaf=DeltaT}
Given integer $m\geq 0$, let non-negative integer array $\bs r=(r_0,r_1, \cdots, r_n)$ satisfy
$$
r_{n}=0,\;\; r_{n-1}=m,\;\; r_{\ell}\geq 2r_{\ell+1} \; \textrm{ for } \ell=n-2,\ldots, 0.
$$
Let $k\geq 2r_0+1 \geq 2^n m + 1$. For $\ell = 1,\dots, n-1,$
\begin{equation}\label{eq:Deltaf=DeltaT}
D(f, r_{\ell}) \backslash \left [ \bigcup_{i=0}^{\ell-1}\bigcup_{e\in \Delta_{i}(f)}D(e, r_{i}) \right ] = D(f, r_{\ell}) \backslash \left [ \bigcup_{i=0}^{\ell-1}\bigcup_{e\in \Delta_{i}(T)}D(e, r_{i}) \right ] .
\end{equation}
\end{lemma}
\begin{proof}
In \eqref{eq:Deltaf=DeltaT}, the relation $\supseteq$ is obvious as $\Delta_i(f)\subseteq \Delta_i(T)$. 

To prove $\subseteq$, it suffices to show for $\alpha \in D(f, r_{\ell}) \backslash \left [ \bigcup_{i=0}^{\ell-1}\bigcup_{e\in \Delta_{i}(f)}D(e, r_{i}) \right ]$, it is not in $D(e,r_i)$ for $e\in \Delta_i(T)$ and $e\not\in\Delta_i(f)$. 

By definition, 
$$
|\alpha_{f^*}|\leq r_{\ell},\; |\alpha_{e}|\leq k - r_{i}-1 \; \textrm{ for all }e\in\Delta_i(f), i=0,\ldots,\ell-1.
$$
For each $e\in\Delta_i(T)$ but $e\not\in\Delta_i(f)$, the dimension of the intersection $e \cap f$ is at most $i-1$. It follows from $r_{j}\geq 2r_{j+1}$ and $k\geq 2r_0+1$ that: when $i>0$,
$$
|\alpha_{e}|=|\alpha_{e\cap f}|+|\alpha_{e\cap f^*}|\leq k - r_{i-1}-1+r_{\ell}\leq k - r_{i}-1,
$$
and when $i=0$,
$$
|\alpha_{e}|=|\alpha_{e\cap f^*}|\leq r_{\ell}\leq k - r_{i}-1.
$$
So $|\alpha_{e^*}| > r_i$. We conclude that $\alpha \not\in D(e, r_i)$ for all $e\in\Delta_i(T)$ and \eqref{eq:Deltaf=DeltaT} follows. 
\end{proof}

We are in the position to present our main result. 
\begin{theorem}\label{th:decT}
Given integer $m\geq 0$, let non-negative integer array $\bs r=(r_0,r_1, \cdots, r_n)$ satisfy
$$
r_{n}=0,\;\; r_{n-1}=m,\;\; r_{\ell}\geq 2r_{\ell+1} \; \textrm{ for } \ell=n-2,\ldots, 0.
$$
Let $k\geq 2r_0+1 \geq 2^n m + 1$. Then we have the following direct decomposition of the simplicial lattice  on an $n$-dimensional simplex $T$:
\begin{align}\label{eq:smoothdecnd}
  \mathbb T^{n}_k(T) = \Oplus_{\ell = 0}^{n}\Oplus_{f\in \Delta_{\ell}(T)} S_{\ell}(f),
\end{align}
where
\begin{align*}
S_0(\texttt{v}) &=  D(\texttt{v}, r_0), \\
S_{\ell}(f) &= D(f, r_{\ell}) \backslash \left [ \bigcup_{i=0}^{\ell-1}\bigcup_{e\in \Delta_{i}(f)}D(e, r_{i}) \right ], \; \ell = 1,\dots, n-1, \\
S_n(T) & = \mathbb T^{n}_k(T) \backslash  \left [  \bigcup_{i=0}^{n-1}\bigcup_{f\in \Delta_{i}(T)}D(f, r_{i}) \right ].
\end{align*}
Consequently we have the following geometric decomposition of $\mathbb P_{k}(T)$
\begin{equation}\label{eq:PrSdec}
\mathbb P_{k}(T) = \Oplus_{\ell = 0}^{n} \Oplus_{f\in \Delta_{\ell}(T)} \mathbb P_k(S_{\ell}(f)).
\end{equation}
\end{theorem}
\begin{proof}
First we show that the sets  $\{S_{\ell}(f), f\in\Delta_{\ell}(T), \ell=0,\ldots,n\}$ are disjoint.
%Hence we acquire
%\begin{equation}\label{eq:SlDrmD}
%S_{\ell}(f) = D(f, r_{\ell}) \backslash \left [ \bigcup_{i=0}^{\ell-1}\bigcup_{e\in \Delta_{i}(T)}D(e, r_{i}) \right ] \quad\textrm{ for } \ell = 1,\dots, n-1.
%\end{equation}
Take two vertices $\texttt{v}_1, \texttt{v}_2\in \Delta_0(T)$. For $\alpha\in D(\texttt{v}_1, r_0)$, we have $\alpha_{\texttt{v}_1} \geq k - r_0$. As $\texttt{v}_1\subseteq \texttt{v}_2^*$ and $k\geq 2r_0+1$, $|\alpha_{\texttt{v}_2^*}|\geq \alpha_{\texttt{v}_1}\geq k-r_0\geq r_0+1$, i.e., $\alpha \notin D(\texttt{v}_2, r_0)$. Hence $\{S_{0}(\texttt{v}), \texttt{v}\in\Delta_{0}(T)\}$ are disjoint and $\Oplus_{\texttt{v}\in\Delta_{0}(T)} S_{0}(\texttt{v})$ is a disjoint union.
By Lemma \ref{lm:disjoint} and \eqref{eq:Deltaf=DeltaT}, we know $\{S_{\ell}(f), f\in\Delta_{\ell}(T), \ell=0,\ldots,n\}$ are disjoint.

Next we inductively prove 
\begin{equation}\label{eq:USequalUD}
\Oplus_{i = 0}^{\ell}\Oplus_{f\in \Delta_{i}(T)} S_{i}(f)=\bigcup_{i=0}^{\ell}\bigcup_{f\in \Delta_{i}(T)}D(f, r_{i}) \quad \textrm{ for }\; \ell=0,\ldots, n-1.
\end{equation}
Obviously \eqref{eq:USequalUD} holds for $\ell=0$. Assume \eqref{eq:USequalUD} holds for $\ell<j$. Then
\begin{align*}
&\quad\Oplus_{i = 0}^{j}\Oplus_{f\in \Delta_{i}(T)} S_{i}(f)= \Oplus_{f\in \Delta_{j}(T)} S_{j}(f)\;\oplus \;\bigcup_{i=0}^{j-1}\bigcup_{e\in \Delta_{i}(T)}D(e, r_{i}) \\
&= \Oplus_{f\in \Delta_{j}(T)}\left(D(f, r_{j}) \backslash \left [ \bigcup_{i=0}^{j-1}\bigcup_{e\in \Delta_{i}(T)}D(e, r_{i}) \right ]\right)\;\oplus \;\bigcup_{i=0}^{j-1}\bigcup_{e\in \Delta_{i}(T)}D(e, r_{i}) \\
&=\bigcup_{i=0}^{j}\bigcup_{f\in \Delta_{i}(T)}D(f, r_{i}).
\end{align*}
By induction, \eqref{eq:USequalUD} holds for $\ell=0,\ldots, n-1$. Then \eqref{eq:smoothdecnd} is true from the definition of $S_n(T)$ and \eqref{eq:USequalUD}. 
%Finally, it follows from \eqref{eq:USequalUD} that
%\begin{align*}
%  \mathbb T^{n}_k(T)=D(T, r_{n}) =\bigcup_{\ell=0}^{n}\bigcup_{f\in \Delta_{\ell}(T)}D(f, r_{\ell})= \Oplus_{\ell = 0}^{n}\Oplus_{f\in \Delta_{\ell}(T)} S_{\ell}(f).
%\end{align*}
%i.e. \eqref{eq:smoothdecnd}.
\end{proof}

We can write out the inequality constraints in $S_{\ell}(f)$. For $\ell = 1,\dots, n$,
\begin{equation}\label{eq:Slfineqlty}
S_{\ell}(f) =  \{\alpha\in\mathbb T_{k}^{n}: |\alpha_{f^*}|\leq r_{\ell}, |\alpha_{e}|\leq k - r_{i}-1, \forall e\in\Delta_i(f), i=0,\ldots,\ell-1\}.
\end{equation}
For $\alpha\in S_{\ell}(f)$, by Lemma~\ref{lm:disjoint} we also have $\alpha\not\in D(\tilde f, r_{\ell})$ for $\tilde f \in \Delta_{\ell} (T)\backslash\{f\}$, i.e.
\begin{equation}\label{eq:Slfineqlty2}
|\alpha_{\tilde f}|\leq k - r_{\ell}-1 \quad\forall~\tilde f \in \Delta_{\ell} (T)\backslash\{f\}.
\end{equation}
From the implementation point of view, the index set $S_{\ell}(f)$ can be found by a logic array and set the entry as true when the distance constraint holds. 

\subsection{Decomposition of degree of freedoms}
Recall that $L(f,s) = \{ \alpha \in \mathbb T^{n}_k, \dist(\alpha,f) = s\}$ consists of lattice nodes $s$ away from $f$. The following unisolvence is the generalization of Lemma \ref{lem:vertexunisolvence} from a vertex to a sub-simplex $f$. 
\begin{lemma}
Let $\ell=0,\ldots, n-1$ and $s\leq r_{\ell}$ be a non-negative integer. Given $f\in \Delta_{\ell}(T)$, let $n_f = \{n_f^1, n_f^2, \ldots, n_f^{n - \ell}\}$ be $n- \ell$ vectors spanning the normal plane of $f$. The polynomial space $\mathbb P_k(S_{\ell}(f)\cap L(f,s))$ is uniquely determined by DoFs
\begin{equation}\label{eq:normalDof}
\int_f  \frac{\partial^{\beta} u}{\partial n_f^{\beta}} \, \lambda_f^{\alpha_f} \dd s \quad \forall~\alpha\in S_{\ell}(f), |\alpha_f| = k - s, \beta \in \mathbb N^{1:n-\ell}, |\beta | = s.
\end{equation}
\end{lemma}
\begin{proof}
A basis of $\mathbb P_k(S_{\ell}(f)\cap L(f,s))$ is
$\{ \lambda^{\alpha} = \lambda_f^{\alpha_f}\lambda_{f^*}^{\alpha_{f^*}} , \alpha \in S_{\ell}(f), | \alpha_{f^*} | = s \}$ and thus the dimensions match (by mapping $\alpha_{f^*}$ to $\beta$).

We choose a basis of the normal plane $\{n_f^1, n_f^2, \ldots, n_f^{n - \ell}\}$ s.t. it is dual to the vectors $\{ \nabla \lambda_{f^*(1)}, \nabla \lambda_{f^*(2)}, \ldots, \}$, i.e., $\nabla \lambda_{f^*(i)}\cdot n_f^j = \delta_{i,j}$ for $i,j=1,\ldots, n - \ell$. Then we have the duality
\begin{equation}\label{eq:Dbeta}
\frac{\partial^{\beta} }{\partial n_f^{\beta}} (\lambda_{f^*}^{\alpha_{f^*}}) = \beta!\delta( \alpha_{f^*}, \beta), \quad \alpha_{f^*}, \beta\in \mathbb N^{1:n-\ell}, |\alpha_{f^*}| =|\beta| = s,
\end{equation}
which can be proved easily by induction on $s$. When $T$ is the reference simplex $\hat T$, $\lambda_i = x_i$ and $\nabla \lambda_i = - \bs{e}_i$, \eqref{eq:Dbeta} is the calculus result $D^{\beta}_{n_f}\bs x_{f^*}^{\alpha_{f^*}}=\beta!\delta( \alpha_{f^*}, \beta)$. 

Assume $u = \sum c_{\alpha_f, \alpha_{f^*}}  \lambda_f^{\alpha_f}\lambda_{f^*}^{\alpha_{f^*}} \in \mathbb P_k(S_{\ell}(f)\cap L(f,s))$. If the derivative is not fully applied to the component $\lambda_{f^*}^{\alpha_{f^*}}$, then there is a term $\lambda_{f^*}^{\gamma}$ with $|\gamma| > 0$ left and $\lambda_{f^*}^{\gamma}|_f = 0$. So for any $\beta \in \mathbb N^{1:n-\ell}$ and $|\beta | = s$,
\begin{equation}\label{eq:normalderivative}
\frac{\partial^{\beta} u }{\partial n_f^{\beta}} |_f = \beta! \sum_{\alpha\in S_{\ell}(f),|\alpha_f| = k- s} c_{\alpha_f, \beta}  \lambda_f^{\alpha_f}.
\end{equation}
The vanishing DoF \eqref{eq:normalDof} implies $\sum\limits_{ \alpha\in S_{\ell}(f),|\alpha_f| = k- s} c_{\alpha_f, \beta}  \lambda_f^{\alpha_f}|_f=0$.
Hence $c_{\alpha_f, \beta} = 0$ for all $|\alpha_f| = k- s, \alpha\in S_{\ell}(f)$. As $\beta$ is arbitrary, we conclude all coefficients $ c_{\alpha_f, \alpha_{f^*}} = 0$ and thus $u = 0$.
%Notice that $D^{\beta}\lambda^{\alpha}|_f = c\, \lambda_f^{\alpha_f}$ as $\lambda_{f^*}|_f = 0$. Then $\frac{\partial^{\beta} u}{\partial n_f^{\beta}}|_f\in \spa \{ \lambda_f^{\alpha_f}, \alpha_f\in S_{\ell}(f)\cap L_k(f)\}$ and vanishing \eqref{eq:normalDof} implies $\frac{\partial^{\beta} u}{\partial n_f^{\beta}}|_f = 0$.
\end{proof}

%\begin{remark}\rm
%Based on \eqref{eq:normalderivative}, the DoFs can be replaced by node values
%$$
%\frac{\partial^{\beta} u}{\partial n_f^{\beta}}(x(\alpha_f)) \quad \forall~\alpha\in S_{\ell}(f), |\alpha_f| = k - s, \beta \in \mathbb N^{1:n-\ell}, |\beta | = s.
%$$
%Here $x(\alpha_f)$ is the geometric embedding using $\alpha_f$ only and thus $x(\alpha_f)\in f$.
%%Other dual basis $\mathbb P_k(S_{\ell}(f)\cap L(f,s))$ can be also used.
%\end{remark}

For $u\in\mathbb P_k(S_{\ell}(f)\cap L(f,s))$ and $\beta\in\mathbb N^{1:n-\ell}$ with $|\beta|<s$, by Lemma~\ref{lm:derivative}, $\frac{\partial^{\beta} u}{\partial n_f^{\beta}}|_f=0$.
Applying the operator $\frac{\partial^{\beta} (\cdot)}{\partial n_f^{\beta}}|_f$ to the direct decomposition $\mathbb P_k(S_{\ell}(f)) = \Oplus_{s=0}^{r_{\ell}} \mathbb P_k(S_{\ell}(f)\cap L(f,s))$ will possess a block lower triangular structure and leads to the following unisolvence result.

\begin{lemma}\label{lem:faceunisolvence}
Let $\ell=0,\ldots, n-1$. The polynomial space $\mathbb P_k(S_{\ell}(f))$ is uniquely determined by DoFs
\begin{equation*}
\int_f  \frac{\partial^{\beta} u}{\partial n_f^{\beta}} \, \lambda_f^{\alpha_f} \dd s \quad \forall~\alpha\in S_{\ell}(f), |\alpha_f| = k - s, \beta \in \mathbb N^{1:n-\ell}, |\beta | = s, s=0,\ldots, r_{\ell}.
\end{equation*}
\end{lemma}

Together with decomposition \eqref{eq:PrSdec} of the polynomial space, we obtain the following result.
\begin{theorem}\label{th:localPrCm}
 Given integer $m\geq 0$, let non-negative integer array $\bs r=(r_0,r_1, \cdots, r_n)$ satisfy
$$
r_{n}=0,\;\; r_{n-1}=m,\;\; r_{\ell}\geq 2r_{\ell+1} \; \textrm{ for } \ell=n-2,\ldots, 0.
$$
Let $k\geq 2r_0+1 \geq 2^n m + 1$. Then the shape function $\mathbb P_k(T)$ is uniquely determined by the following DoFs
\begin{align}
\label{eq:C1nd0}
D^{\alpha} u (\texttt{v}) & \quad \alpha \in \mathbb N^{1:n}, |\alpha | \leq  r_0, \texttt{v}\in \Delta_0(T),\\
\label{eq:C1nd2}
\int_f \frac{\partial^{\beta} u}{\partial n_f^{\beta}}  \, \lambda_f^{\alpha_f} \dd s & \quad \alpha\in S_{\ell}(f), |\alpha_f| = k - s, \beta \in \mathbb N^{1:n-\ell}, |\beta | = s,\\
&\quad f\in \Delta_{\ell}(T), \ell =1,\ldots, n-1, s=0,\ldots, r_{\ell}, \notag \\
\label{eq:C1nd3}
\int_T u \lambda^{\alpha} \dx & \quad \alpha \in S_n(T).
% \in \mathbb N^{0:3}_{r-4(m+1)}, \alpha \leq r - r_0-m  - 2 \textrm{ and } \\
%&\quad \alpha_i+\alpha_j \leq r  - r_1 -2m-3 \textrm{ for } 0\leq i<j\leq 3. \notag
\end{align}
\end{theorem}
\begin{proof}
Thanks to the decomposition \eqref{eq:PrSdec}, the dimensions match.
Take $u\in\mathbb P_k(T)$ satisfy all the DoFs \eqref{eq:C1nd0}-\eqref{eq:C1nd3} vanish. We are going to show $u=0$.

For $\alpha\in S_{\ell}(f)$ and $e\in\Delta_i(T)$ with $i\leq\ell$ and $e\neq f$,
by \eqref{eq:Slfineqlty} and \eqref{eq:Slfineqlty2} we have $|\alpha_{e^*}|\geq r_{i}+1$, hence $\frac{\partial^{\beta}\lambda^{\alpha}}{\partial n_e^{\beta}}|_e=0$ for $\beta\in\mathbb N^{1:n-i}$ with $|\beta|\leq r_i$.
Again this tells us that applying the operator $\frac{\partial^{\beta}(\cdot)}{\partial n_f^{\beta}}|_f$ to the direct decomposition $\mathbb P_{k}(T) = \Oplus_{\ell = 0}^{n} \Oplus_{f\in \Delta_{\ell}(T)} \mathbb P_k(S_{\ell}(f))$ will produce a block lower triangular structure. Then apply Lemma~\ref{lem:faceunisolvence}, we conclude $u\in\mathbb P_k(S_{n}(T))$, which together with the vanishing DoF \eqref{eq:C1nd3} gives $u=0$.
\end{proof}
% We use the decomposition $D(f,r) = \Oplus_{k=0}^{r_{\ell}} L_k(f)$ and due to Lemma~\ref{lm:derivative},
% \LC{Write one more paragraph for $|\beta|\leq r_{\ell}$ using the lower triangular structure.}

\begin{remark}\rm
For $\alpha\in S_{\ell}(f)$, by \eqref{eq:Slfineqlty} we have $|\alpha_{e}|\leq k - r_{\ell-1}-1$ for all $e\in\Delta_{\ell-1}(f)$, then $\alpha_f\geq r_{\ell-1}+1-|\alpha_{f^*}|$, and
$$
\lambda^{\alpha}=\lambda_{f^*}^{\alpha_{f^*}}\lambda_f^{\alpha_f}=\lambda_{f^*}^{\alpha_{f^*}}\lambda_f^{r_{\ell-1}+1-|\alpha_{f^*}|}\lambda_f^{\alpha_f-(r_{\ell-1}+1)+|\alpha_{f^*}|}.
$$
As the two-dimensional case, using $\alpha_f-(r_{\ell-1}+1)+|\alpha_{f^*}|$ as the new index, DoFs \eqref{eq:C1nd2}-\eqref{eq:C1nd3} can be replaced by
\begin{align*}
\int_f \frac{\partial^{\beta} u}{\partial n_f^{\beta}}  \, \lambda_f^{\alpha} \dd s & \quad \beta \in \mathbb N^{1:n-\ell}, |\beta | = s, s=0,\ldots, r_{\ell}, \;\;\alpha\in\mathbb T_{k-(\ell+1)(r_{\ell-1}+1)+\ell s}^{\ell}, \\
&\quad |\alpha_{e}|\leq k - r_{i}-1-(i+1)(r_{\ell-1}+1-s), \forall e\in\Delta_i(f), i=0,\ldots,\ell-2, \\
&\quad f\in \Delta_{\ell}(T), \ell =1,\ldots, n-1, \notag \\
\int_T u \lambda^{\alpha} \dx & \quad \alpha \in \mathbb T_{k-(n+1)(m+1)}^{n},\\
&\quad |\alpha_{e}|\leq k - r_{i}-1-(i+1)(m+1), \forall e\in\Delta_i(T), i=0,\ldots,n-2.
% \in \mathbb N^{0:3}_{r-4(m+1)}, \alpha \leq r - r_0-m  - 2 \textrm{ and } \\
%&\quad \alpha_i+\alpha_j \leq r  - r_1 -2m-3 \textrm{ for } 0\leq i<j\leq 3. \notag
\end{align*}
Namely we can remove bubble functions in the test function space. 
\end{remark}

\subsection{Smooth finite elements in arbitrary dimension}
Given a triangulation $\mathcal T_h$, the finite element space is obtained by asking the DoFs depending on the sub-simplex only.

\begin{theorem}\label{th:C1Rn}
 Given integer $m\geq 0$, let non-negative integer array $\bs r=(r_0,r_1, \cdots, r_n)$ satisfy
$$
r_{n}=0,\;\; r_{n-1}=m,\;\; r_{\ell}\geq 2r_{\ell+1} \; \textrm{ for } \ell=n-2,\ldots, 0.
$$
Let $k\geq 2r_0+1 \geq 2^n m + 1$. The following DoFs
\begin{align}
\label{eq:C1Rnd0}
D^{\alpha} u (\texttt{v}) & \quad \alpha \in \mathbb N^{1:n}, |\alpha | \leq  r_0, \texttt{v}\in \Delta_0(\mathcal T_h),\\
\label{eq:C1Rndf}
\int_f \frac{\partial^{\beta} u}{\partial n_f^{\beta}}  \, \lambda_f^{\alpha_f} \dd s & \quad \alpha\in S_{\ell}(f), |\alpha_f| = k - s, \beta \in \mathbb N^{1:n-\ell}, |\beta | = s, s=0,\ldots, r_{\ell},\\
&\quad f\in \Delta_{\ell}(\mathcal T_h), \ell =1,\ldots, n-1, \notag \\
\label{eq:C1RndT}
\int_T u \lambda^{\alpha} \dx & \quad \alpha \in S_n(T), T\in \mathcal T_h,
% \in \mathbb N^{0:3}_{r-4(m+1)}, \alpha \leq r - r_0-m  - 2 \textrm{ and } \\
%&\quad \alpha_i+\alpha_j \leq r  - r_1 -2m-3 \textrm{ for } 0\leq i<j\leq 3. \notag
\end{align}
will define a finite element space
$$
V_h = \{ u \in C^{m}(\Omega) \mid \text{ DoFs }\eqref{eq:C1Rnd0}-\eqref{eq:C1Rndf} \text{ are single valued}, u|_{T}\in \mathbb P_k(T), \forall T\in \mathcal T_h \}.
$$
\end{theorem}
\begin{proof}
Restricted to one simplex $T$, by Theorem \ref{th:localPrCm}, DoFs \eqref{eq:C1Rnd0}-\eqref{eq:C1RndT}  will define a function $u$ s.t. $ u|_{T}\in \mathbb P_k(T)$. We only need to verify $u\in C^{m}(\Omega)$. 
It suffices to prove $\frac{\partial^{i} u}{\partial n_F^{i}}|_F \in \mathbb P_{k-i}(F)$, for all $i=0,\ldots, m$ and all $F\in \Delta_{n-1}(T)$, are uniquely determined by \eqref{eq:C1Rnd0}-\eqref{eq:C1Rndf} on $F$. 

Let $w=\frac{\partial^{i} u}{\partial n_F^{i}}|_F\in \mathbb P_{k-i}(F)$. Consider the modified index sequence $\bs r_F^{i} = (r_0 - i, r_1 - i, \ldots, r_{n-2} - i, 0)$ and degree $k^i = k - i$. Then $k^i, \bs r_F^{i}$ satisfies the condition in Theorem \ref{th:decT} and we obtain a direct decomposition of $\mathbb T^{n-1}_{k-i}(F) = \Oplus_{\ell = 0}^{n-1}\Oplus_{f\in \Delta_{\ell}(F)} S_{\ell}^F(f),$ where
\begin{align*}
S_0^F(\texttt{v}) &=  D(\texttt{v}, r_0-i)\cap \mathbb T^{n-1}_{k-i}(F), \\
S_{\ell}^F(f) &= (D(f, r_{\ell}-i)\cap \mathbb T^{n-1}_{k-i}(F)) \backslash \left [  \Oplus_{i = 0}^{\ell - 1}\Oplus_{e\in \Delta_{i}(F)} S_{i}^F(e) \right ], \; \ell = 1,\dots, n-2, \\
S_{n-1}^F(F) & = \mathbb T^{n-1}_{k-i}(F) \backslash \left [  \Oplus_{\ell = 0}^{n-2}\Oplus_{f\in \Delta_{\ell}(F)} S_{\ell}^F(f)\right ].
\end{align*}
The DoFs \eqref{eq:C1Rnd0}-\eqref{eq:C1Rndf} related to $w$ are
\begin{align*}
D_F^{\alpha} w (\texttt{v}) & \quad \alpha \in \mathbb N^{1:n-1}, |\alpha | \leq  r_0-i, \texttt{v}\in \Delta_0(F),\\
\int_f \frac{\partial^{\beta} w}{\partial n_{F,f}^{\beta}}  \, \lambda_f^{\alpha_f} \dd s & \quad \alpha\in S_{\ell}^F(f), |\alpha_f| = k-i - s, \beta \in \mathbb N^{1:n-1-\ell}, |\beta| = s, \\
&\quad f\in \Delta_{\ell}(F), \ell =1,\ldots, n-2, s=0,\ldots, r_{\ell}-i, \notag \\
\int_F w \lambda^{\alpha} \dx & \quad \alpha \in S_{n-1}^F(F),
\end{align*}
where $D_Fw$ is the tangential derivatives of $w$, $n_{F,f}$ is the normal vector of $f$ but tangential to $F$.
Clearly the modified sequence $\bs r^F_{i}$ still satisfies constraints required in Theorem \ref{th:localPrCm}. We can apply Theorem \ref{th:localPrCm} with the shape function space $\mathbb P_{k-i}(F)$ to conclude $w$ is uniquely determined on $F$. Thus the result $u\subset C^m(\Omega)$ follows.
% Given an edge $e\in \Delta_1(\mathcal T_h)$ and $\beta \in \mathbb N^{1:n-1}$, let $w = \frac{\partial^{\beta} u}{\partial n_e^{\beta}}|_e \in \mathbb P_{r-|\beta|}(e)$. Then $\partial_e^{(j)}w(\texttt{v})$ for $j=0,1,\ldots, r_0 - |\beta|$ is determined by \eqref{eq:C1Rnd0} and
% $$
% \int_e w\lambda_e^{\alpha_e}\dd s\quad \alpha_e\in S_1(e), |\alpha_e| = k - |\beta|, \alpha_e \geq r_0 - |\beta| + 1
% $$
% is determined by \eqref{eq:C1Rndf}. Then apply Theorem \ref{th:localPrCm} to $e$, $w|_e \in \mathbb P_{r-|\beta|}(e)$ is uniquely determined. We thus verified the statement for $\ell = 1$.
% by 1-D Hermite interpolation \mnote{ modify the previous one by adding the inequality constraint in the dof},
% Assume \LC{more. Just rewrite DoF with different indexes and apply Theorem \ref{th:localPrCm}. Verify the modified sequence $\bs r^{i} = (r_0 - i, r_1 - i, \ldots, r_{n-2} - i, 0)$ still satisfies the constraint.}.
\end{proof}

Counting the dimension of $V_h$ is hard and not necessary. The cardinality of $S_{\ell}(f)$ is difficult to determine due to the inequality constraints. In the implementation, compute the distance of lattice nodes to sub-simplexes and use a logic array to find out $S_{\ell}(f)$. Our geometric decomposition using the distance not only simplifies the construction but also provide an easy way of implementation.

\section*{Acknowledgement}
The authors would like to thank Dr. Shuhao Cao in Washington University, St. Louis for providing figures.

\bibliographystyle{abbrv}
\bibliography{refgeodecomp}
\end{document}